%% file: arxiv_main.tex
\documentclass[11pt,a4paper,final]{amsart}
\usepackage[DIV=14,paper=a4,BCOR=14mm,headinclude]{typearea}
\usepackage[british]{babel}
\usepackage{amssymb}
\usepackage{graphicx}
\usepackage{color}
\usepackage{tikz}
\usepackage{pgfplots}
\usepackage{enumitem}
\usepackage{hyperref}
\usepackage{fp-eval}
\usepackage{esint}
\usepackage{caption}
\usepackage{todonotes}
\usepackage{mathabx}
\usepackage{booktabs}
\usepackage{ngerman}
\usepackage{subcaption}
\usepackage{multirow}
\definecolor{darkblue}{rgb}{0,0,.7}
\hypersetup{colorlinks=true,linkcolor=darkblue,citecolor=darkblue}

\newlist{alphenum}{enumerate}{1}
\setlist[alphenum]{fullwidth,label={(\alph*)}}

\allowdisplaybreaks[0]
\theoremstyle{definition}
\newtheorem{theorem}{Theorem}[section]

\newtheorem{remark}[theorem]{Remark}
\newtheorem{lemma}[theorem]{Lemma}

\numberwithin{figure}{section}
\numberwithin{table}{section}
\numberwithin{equation}{section}

\DeclareMathOperator*{\argmin}{argmin}
\renewcommand{\vec}[1]{\boldsymbol{#1}}

\tikzset{every picture/.append style={font=\normalsize}}
\pgfplotsset{
  axis background/.style={fill=black!5!white},
  minor grid style={densely dotted,semithick},
  compat=1.14}
\tikzset{slopetriangle/.style={
  bottom color=white,
  top color=white,
  draw=black
}}

\begin{document}
\date{\today}

\title[divergence-free reconstruction on polygons]{Divergence-preserving reconstructions on polygons and a really pressure-robust virtual element method for the Stokes problem}

\author{D.~Frerichs, C.~Merdon}
\maketitle

\begin{abstract}
Non divergence-free discretisations for the incompressible Stokes problem may suffer from a lack of pressure-robustness characterised by large discretisations errors due to irrotational forces in the momentum balance.
This paper argues that also divergence-free virtual element methods (VEM) on polygonal meshes are not really pressure-robust as long as the right-hand side is not discretised in a careful manner. To be able to evaluate the right-hand side for the
testfunctions, some explicit interpolation of the virtual testfunctions is needed that can be evaluated
pointwise everywhere. The standard discretisation via an \(L^2\)-bestapproximation does not
preserve the divergence and so destroys the orthogonality
between divergence-free testfunctions and possibly eminent gradient forces
in the right-hand side. To repair this orthogonality and restore pressure-robustness another
divergence-preserving reconstruction is suggested
based on Raviart--Thomas approximations on local subtriangulations of the polygons.
All findings are proven theoretically and are demonstrated numerically in two dimensions. 
The construction is also interesting for hybrid high-order methods on polygonal or polyhedral meshes.
\end{abstract}

\section{Introduction}
Recently, the mathematical community became interested in flexible approximation methods on polygonal or polyhedral meshes. For the Stokes problem, several approaches are available, see e.g.\ \cite{MR2465455,DIPIETRO20151,MR2285842,MR3266951,MR3407240,MR3502564} and the references therein. 
One very popular and elegant approach is the virtual element method  \cite{MR3796371,BeiraodaVeiga2019} that
preserve the $H^1$-conformity and the divergence constraint of the velocity field on the discrete level on polygonal meshes. Usually, conforming divergence-free methods are also pressure-robust as any divergence-free function is orthogonal against (pressure) gradients that appear in the momentum balance \cite{MR3683678}, in particular in the right-hand side.

However, the fact that the virtual test functions are only known at the degrees of freedom complicates the discretisation
of the right-hand side. Consequently, in the context of virtual element methods the right-hand side functional
\begin{align*}
  F(\vec{v}_h) := \int_\Omega \vec{f} \cdot \vec{v}_h \, \mathrm{d}x
\end{align*}
in general cannot be evaluated exactly and has to be approximated. To do so, the information on the ansatz functions allows to compute an $\vec{L}^2$ bestapproximation $\vec{\pi}_k$ of a certain degree $k$. This leads to the approximative right-hand side
\begin{align*}
  F_h(\vec{v}_h) := \int_\Omega \vec{f} \cdot \vec{\pi}_k \vec{v}_h \, \mathrm{d}x.
\end{align*}
In the a priori error estimate for the velocity error \(\| \nabla(\vec{u}
- \vec{u}_h) \|_{L^2}\) an additional discretisation error  pops up that
can be quantified by the dual norm of \(\vec{F} - \vec{F}_h\) with respect to the divergence-free VEM subspace \(\vec{V}_{0,h}\), i.e.
\begin{align*}
  \| \vec{F} - \vec{F}_h \|_{\vec{V}_{0,h}^\star}
  := \sup_{\vec{v}_h \in \vec{V}_{0,h} \setminus \lbrace 0 \rbrace}
  \frac{F(\vec{v}_h) - F_h(\vec{v}_h)}{\| \nabla \vec{v}_h \|_{L^2}}
  & \lesssim \mathcal{O}(h^{k+2}) \vert \vec{f} \vert_{H^{k+1}}
\end{align*}
Since this consistency error enters the a priori velocity estimate with the inverse of the viscosity \(1/\nu\), the velocity error
might be large in case of large complicated pressures $p/\nu$, e.g.\ \(\vec{f} = \nabla p\) in the worst case. A pressure-robust discretisation would be pressure-independent and locking-free for \(\nu \rightarrow 0\) in the sense of \cite{Babuska:Suri:SINUM:92,MR3743746}, see \cite{MR3683678,MR3564690,MR3481034,gauger:linke:schroeder:2019} for more details on pressure-robustness and why it is important.
Although an enhanced version of the VEM achieves a discretisation error in the right-hand side of higher order, the method is only asymptotically pressure-robust for \(h \rightarrow 0\), but still can show large errors on coarser meshes which is demonstrated in the numerical examples below.

This contribution argues that uniform pressure-robustness, meaning on any mesh, can only be attained by an $\vec{H}(\mathrm{div},\Omega)$-conforming
interpolation $\Pi$ that preserves the divergence of the virtual test functions.
On triangles, such an interpolation is given by a standard Raviart-Thomas
interpolation in the spirit of \cite{MR3133522,MR3656505,MR3460110}, that also
can be evaluated for the virtual ansatz functions of \cite{MR3796371} as studied
in the master's thesis \cite{masterthesis} for order \(k=2\) and proven here for arbitrary order $k$. On polygons, the same idea can be exploited on a subtriangulation of the polygon and requires to solve small local Dirichlet boundary value problems for each virtual test function on each polygon.
This leads to the alternative right-hand side discretisation
\begin{align*}
  \vec{F}_{\mathrm{RT}_{k-1}}(\vec{v}_h)
  & := \int_\Omega \vec{f} \cdot I_{\mathrm{RT}_{k-1}} \vec{v}_h \, \mathrm{d}x\\
  & = \int_\Omega \mathbb{P} \vec{f} \cdot I_{\mathrm{RT}_{k-1}} \vec{v}_h \, \mathrm{d}x \quad \text{for } \vec{v}_h \in \vec{V}_{0,h}
\end{align*}
and the corresponding discretisation error can be estimated by
\begin{align*}
  \| \vec{F} - \vec{F}_{\mathrm{RT}_{k-1}} \|_{\vec{V}_{0,h}^\star}
  \lesssim \| h_\mathcal{T}(\mathbb{P}\vec{f} - \vec{\pi}_{k-2}(\mathbb{P}\vec{f}) )\|_{L^2}
  & \lesssim \nu \| h^k_\mathcal{T} D^{k-1} \Delta \vec{u} \|_{L^2}.
\end{align*}
Here, \(\mathbb{P} \vec{f} \in \vec{L}^2(\Omega)\) is the (divergence-free) Helmholtz projector of
\(\vec{f}\), that can be identified as \linebreak\(\mathbb{P} \vec{f} = - \nu \Delta \vec{u}\) when testing with divergence-free test functions, see \cite{alex2019pressurerobustness} for details. Surprisingly, for the virtual element method of order $k=2$, also a lowest order Raviart--Thomas interpolation \( I_{\mathrm{RT}_0}\) seems enough to preserve the optimal velocity convergence order, i.e. it holds the estimate
\begin{align*}
  \| \vec{F} - \vec{F}_{\mathrm{RT}_{0}} \|_{\vec{V}_{0,h}^\star}
  \lesssim \| h^2_\mathcal{T} \mathrm{curl} (\vec{f}) \|_{L^2}
\end{align*}
but at the price that the pressure error converges only suboptimally with order \(1\). The proof employs techniques from \cite{Lederer2019}.

Finally, we want to stress that the design of the reconstruction operator can be transferred also in the setting of hybrid high order methods on general meshes \cite{DIPIETRO20151} which can be seen as a generalisation of the design in \cite{MR3502564} on simplicial meshes. This observation together with other conclusions are reported at the end of the paper.
Also, although all results are stated in two dimensions, everything can be extended to three dimensions in a straightforward way.

\medskip
The rest of the paper is organised as follows. Section~\ref{sec:Preliminaries} introduces the Stokes model problem and some preliminaries.
Section~\ref{sec:VEM} discusses the classical virtual element discretisation and some improvements invented by the VEM community that already help to repair the lack of pressure-robustness to a certain extent.
Section~\ref{sec:ReconstructionPolygons} observes and proves that
a pressure-robust discretisation on shape-regular polygons is possible with the
help of Raviart--Thomas interpolations which can be computed despite the virtuality of the VEM testfunctions.
The resulting pressure-robust a priori estimates are shown in
Section~\ref{sec:aprioriprobust} as well as the surprising fact that also
a standard Raviart--Thomas interpolation of lower order is enough to keep the
optimal order of convergence for the velocity error.
Section~\ref{sec:Numerics} shows some numerical examples that confirm the theoretical results.
Finally, Section~\ref{sec:Outlook} discusses some generalisations and the relevance of the reconstruction operator for the full Navier-Stokes problem.

\section{Preliminaries}\label{sec:Preliminaries}
This section recalls the Stokes model problem and the Helmholtz--Hodge projector which is an important tool to explain pressure-robustness and
to derive pressure-robust error estimates.

\subsection{Stokes model problem}
Consider some two dimensional Liptschitz domain \(\Omega\)
with boundary \(\partial \Omega\).
The Stokes equations seek some velocity field \(\vec{u} \in \vec{H}^1_0(\Omega)\) and some pressure
field \(p \in L^2_0(\Omega) := \lbrace q \in L^2(\Omega) : \int_\Omega q \, \mathrm{d}x = 0\rbrace \) such that
\begin{align*}
  - \nu \Delta \vec{u} + \nabla p = \vec{f},\quad \text{and}\quad
  \mathrm{div}\,\vec{u} = 0 \quad\quad \text{in }\Omega
\end{align*}
for some given right-hand side \(\vec{f} \in \vec{L}^2(\Omega)\) and positive viscosity \(\nu > 0\).

The weak solution is characterised by
\begin{align*}
  a(\vec{u},\vec{v}) + b(p,\vec{v}) & = F(v) && \text{for all } \vec{v} \in \vec{H}^1_0(\Omega),\\
  					   b(q,\vec{u}) & = 0    && \text{for all } q \in L^2_0(\Omega)
\end{align*}
where
\begin{align*}
  a(\vec{u},\vec{v}) &:= \nu \int_\Omega \nabla \vec{v} : \nabla \vec{u} \, \mathrm{d}x,\\
  b(q,\vec{v}) &:= -\int_\Omega q \mathrm{div} \vec{v} \, \mathrm{d}x,\\
  F(\vec{v}) &:= \int_\Omega \vec{f} \cdot \vec{v} \, \mathrm{d}x.
\end{align*}

From standard saddle point theroy (see e.g.\ \cite{MR3097958}) it
is well known that it exists a unique solution
$(\vec{u},p)\in \vec{H}^1_0(\Omega)\times L^2_0(\Omega)$
to the Stokes equations.

\subsection{Helmholtz--Hodge projector and pressure-robustness}
Recall the $\vec{L}^2$-orthogonal Helmholtz--Hodge decomposition
(see e.g.\ \cite{GirRav-nse})
that decomposes any vector field $\vec{f} \in \vec{L}^2(\Omega)$ uniquely
into 
\begin{align} \label{eqn:HHLdecomposition}
  \vec{f} = 
  \nabla \alpha + \mathbb{P}(\vec{f}),
\end{align}
where $\alpha \in H^1(\Omega)/\mathbb{R}$, and
\[
\mathbb{P}(\vec{f}) 
\in \vec{L}^2_\sigma(\Omega) := \lbrace \vec{w} \in \vec {L}^2(\Omega) : 
(\nabla q, \vec{w}) = 0 \text{ for all } q \in H^1(\Omega) \rbrace.
\]
The latter one is called the Helmholtz--Hodge projector $\mathbb{P}(\vec{f})$
of $\vec{f}$ and is divergence-free. Also note that
\(\mathbb{P}(\nabla q) = 0\) for any \(q \in H^1(\Omega)\).

On the continuous level the \(\nabla \alpha\) part of the right-hand side in the momentum balance of the Stokes equations goes into the pressure \(p\), whereas the Helmholtz-projector determines the velocity. Pressure-robust discretisations respect this balance
and avoid an influence of \(\alpha\) on the velocity \cite{MR3683678,MR3564690,alex2019pressurerobustness}.

Therefore, a pressure-robust discretisation is characterised by a velocity error that is independent of the exact pressure.

\section{Virtual element methods for the Stokes problem}\label{sec:VEM}
This section introduces some notation and the setup of the virtual element method for the Stokes problem as given in \cite{MR3626409}. The last two subsections comment on known a priori estimates and an enhanced version of \cite{MR3796371} that improves the disretisation error of the right-hand side without healing
the lack of pressure-robustness completely.

\subsection{Mesh notation and assumptions}
Throughout the paper, \(\mathcal{T}\) denotes a
decomposition of the domain \(\Omega\subset \mathbb{R}^2\) into
non-overlapping simple polygons \(K\) with
\begin{align*}
  h_K:=\mathrm{diam}(K) \quad\text{and}\quad h:=
\sup_{K\in \mathcal{T}_h} h_K.
\end{align*}
Moreover, \(\mathcal{E}\) denotes the set of faces of the decomposition
$\mathcal{T}$ and \(\mathcal{E}(K)\)
denotes the set of faces of a polygon \(K \in \mathcal{T}\).

For simplicity, \(\mathcal{T}\) is supposed to fulfill the following standard shape regularity
properties, see e.g.\ \cite{BasicVEM,MR3626409}:
There exist two positive constants $\gamma_1,\gamma_2\in\mathbb{R}$, such that each \(K\in\mathcal{T}\) satisfies the assumptions
\begin{itemize}
  \item[(A1)] \(K\) is star-shaped with respect to a ball of radius larger or
    equal to $\gamma_1\,h_K$,
  \item[(A2)] the distance between any two vertices of $K$ is larger or equal to
    $\gamma_2\,h_K$.
\end{itemize}
As usual this shape regularity properties can be weakend a little, see
\cite{BasicVEM,MR3626409}.

References to convergence rates in this paper always are meant with respect to a series of decompositions with uniformly bounded \(\gamma_1,\gamma_2\).
Constants hidden in \(\lesssim\) may depend on these bounds but not on \(h\).

\subsection{Virtual element method}
The virtual element method (VEM) for solving the Stokes problem given in \cite{MR3626409} shall serve as
a starting point for the new pressure-robust version.

For a fixed integer $k\in\mathbb{N}$, on each element $K\in\mathcal{T}$ the local
virtual element spaces are defined by
\begin{align*}
  \vec{V}_h^K & :=\Big\{\,\vec{v}_h\in\vec{H}^1(K)\,:\,\vec{v}_{h}|_{\partial K}\in
    \vec{C}^0(\partial K),\,\vec{v}_{h}|_{E}\in\vec{P}_k(K)\text{ for all
  }E\in\mathcal{E}(K), \\
 & \qquad \qquad -\nu\Delta\vec{v}_h +\nabla
s\in\mathcal{G}_{k-2}(T)^\perp\text{ for some }s\in
L^2(K),\,\mathrm{div}\,\vec{v}_h\in {P}_{k-1}(K)
\,\Big\},\\
 Q_h^K& :=P_{k-1},
\end{align*}
where $P_{k}(K)$ and $\vec{P}_k(K)$ denote the scalar-valued and vector-valued
polynomials of degree at most $k$ on $K$, respectively, and $\mathcal{G}_{k-2}(K)^\perp\subset \vec{P}_{k-2}(K)$ is the $\vec{L}^2$-orthogonal
complement to $\nabla P_{k-1}(K)$.
This means that every vector valued polyonomial $\vec{q}_{k-2}$ of degree at most
$k-2$ can be decomposed into
a gradient and an orthogonal part, i.e.
\begin{align}\label{eq:polyDecomp}
  \vec{q}_{k-2} = \nabla r_{k-1} + \vec{s}_{k-2}^\perp,
\end{align}
where $r_{k-1}\in P_{k-1}$ and $\vec{s}_{k-2}^\perp \in \mathcal{G}_{k-2}^\perp$.

For a given function $\vec{v}_h\in\vec{V}_h^K$ the following degrees of freedom
are chosen:
\begin{itemize}
  \item \textbf{D\textsubscript{V}1}: the values of $\vec{v}_h$ at the vertices
    of the polygon $K$,
  \item \textbf{D\textsubscript{V}2}: the values of $\vec{v}_h$ at $k-1$
    disctinct internal points of every edge $E\in\mathcal{E}(K)$,
  \item \textbf{D\textsubscript{V}3}: the moments
    \begin{align*}
      \int_K \vec{v}_h\cdot \vec{g}_{k-2}^\perp\,\mathrm{d}x  \quad \quad \text{for all
      }\vec{g}_{k-2}^\perp\in\mathcal{G}_{k-2}^\perp,
    \end{align*}
  \item \textbf{D\textsubscript{V}4}: the moments
    \begin{align*}
      \int_K \mathrm{div}\,\vec{v}_h\,q_{k-1}\,\mathrm{d}x  \quad \quad \text{for all
      }q_{k-1}\in {P}_{k-1}/\mathbb{R}.
    \end{align*}
\end{itemize}

In addition to that, the local pressure $q_h\in Q_h^K$ is defined by the degrees of freedom
\begin{itemize}
  \item \textbf{D\textsubscript{Q}}: the moments
    \begin{align*}
      \int_K q_h\,r_{k-1}\,\mathrm{d}x\quad\quad\text{for all }r_{k-1}\in P_{k-1}(K).
    \end{align*}
\end{itemize}

\begin{lemma}
  The degrees of freedom \textbf{D\textsubscript{V}} and
  \textbf{D\textsubscript{Q}} are unisolvent for the virtual space $\vec{V}_h^K$ and
  $Q_h^K$, respectively.
\end{lemma}
\begin{proof}
  See Proposition 3.1 in \cite{MR3626409}.
\end{proof}

The global virtual element spaces are defined as
\begin{align*}
  \vec{V}_h&:=\Big\{\,\vec{v}_h\in\vec{H}^1_0(\Omega) \,:\,
\vec{v}_{h}|_{K}\in\vec{V}_h^K\quad\text{for all }K\in\mathcal{T}\,\Big\}\\
    Q_h &:= \Big\{\, q_h\in L^2(\Omega) \,:\, q_{h}|_{K}\in Q_h^K\quad \text{for
    all }K\in \mathcal{T} \,\Big\},
\end{align*}
with global degrees of freedom as the collection of the local ones, with
appropriate continuity of facial degrees of freedom \textbf{D\textsubscript{V}1}
and \textbf{D\textsubscript{V}2} across polygonal boundaries.

Next, discrete bilinearforms are chosen.
For this purpose, on each $K\in\mathcal{T}$ the energy projection
$\Pi_k^{\nabla,K}:\vec{V}_h^K\rightarrow \vec{P}_k(K)$ is needed,
defined as solution of 
\begin{align*}
  a(\vec{q}_h,\vec{v}_h-\Pi_h^{\nabla,K}\vec{v}_h) &= 0 \quad \quad \text{for all
  }\vec{q}_h\in\vec{P}_k(K)\\
  \vec{\pi}_0(\vec{v}_h-\Pi_h^{\nabla,K}\vec{v}_h) &= 0,
\end{align*}
where \(\vec{\pi}_k\) denotes the piecewise bestapproximation into
  the polyonomials \(\boldsymbol{P}_k\), and locally
$a^K(\vec{u}_h,\vec{v}_h):= \nu\int_K
\nabla\vec{u}_h:\nabla\vec{v}_h\,\mathrm{d}x$ for all
$\vec{u}_h,\vec{v}_h\in\vec{V}_h^K$.

As shown in \cite{MR3626409} the projection $\Pi_k^{\nabla,K}\vec{v}_h$ of any
virtual function $\vec{v}_h\in\vec{V}_h^K$ can be computed using only the degrees of
freedom and it holds the Poincar\'e inequality
\begin{align} \label{eq:approxprop_gradientprojection}
  \| \vec{v}_h - \Pi_h^{\nabla,K}\vec{v}_h \|_{L^2(K)}
  \lesssim h_K \| \nabla \vec{v}_h \|_{L^2(K)}.
\end{align}

The discrete bilinear forms
$a_h^P:\vec{V}_h^K\times\vec{V}_h^K\rightarrow \mathbb{R}$ and
$b_h^P:Q_h^K\times\vec{V}_h^K\rightarrow\mathbb{R}$ are defined by
\begin{align*}
  a_h^K(\vec{u}_h,\vec{v}_h) & :=a\left(\Pi_k^{\nabla,K}\vec{u}_h,\Pi_k^{\nabla,K}\vec{v}_h\right)
    +\nu \mathcal{S}^K\left((I-\Pi_k^{\nabla,K})\vec{u}_h,(I-\Pi_k^{\nabla,K})\vec{v}_h\right),\\
  b_h^K(q_h,\vec{v}_h) & :=b^K(q_h,\vec{v}_h):=
  \int_K q_h\mathrm{div}\vec{v}_h\,\mathrm{d}x
\end{align*}
for all $\vec{u}_h,\vec{v}_h\in\vec{V}_h^K, q_h\in Q_h^K$, where
$S^K:\vec{V}_h^K\times\vec{V}_h^K\rightarrow \mathbb{R}$ is some stability
bilinear form. Possible choices for the stability bilinear form are given for
instance in \cite{ar:stability}. Since the choice of the stability
bilinear form does not matter for our purpose, we simply use
the vector product of the evaluations of the degrees of freedoms
\begin{align*}
  \mathcal{S}^K\left(\vec{u}_h,\vec{v}_h\right)
  = \textbf{D\textsubscript{V}}(\vec{u}_h) \cdot \textbf{D\textsubscript{V}}(\vec{v}_h).
\end{align*}

The global bilinearforms $a_h(\cdot,\cdot)$ and $b_h(\cdot,\cdot)$ are the sums over the local contributions.
The 'classical' discretisation of the VEM (see e.g.\ \cite{MR3626409}) right-hand
side reads
\begin{align*}
  \vec{F}_h(\vec{v}_h) := \int_\Omega \vec{\pi}_{k-2} \vec{f} \cdot \vec{v}_h
  = \int_\Omega \vec{f} \cdot \vec{\pi}_{k-2} \vec{v}_h
\end{align*}
where \(\vec{\pi}_{k-2}\) is the piecewise \(\vec{L}^2\)-bestapproximtion onto the vector-valued polynomials of degree \(k-2\). Later, alternative (pressure-robust) discretisations are introduced. However, we first turn our focus on the
possible a priori error estimates one obtains with this classical choice.

It can be easily checked that all the bilinear forms and the projections can be
evaluated only with the degrees of freedom, see e.g.\ \cite{DASSI2019} for details.
Therefore, the discrete problem reads as follows: Find
$(\vec{u}_h,p_h)\in\vec{V}_h\times Q_h$ such that
\begin{align*}
  a_h(\vec{u}_h,\vec{v}_h) + b_h(p_h,\vec{v}_h) &= F(\vec{v}_h)&&\text{for all
  }\vec{v}_h\in\vec{V}_h\\
  b_h(q_h,\vec{u}_h) &= 0 && \text{for all }q_h\in Q_h.
\end{align*}
The discrete problem has a unique (but virtual) solution which is pointwise divergence free \cite{MR3626409}.

\subsection{A priori error estimates}

This section recalls a priori error estimates for the velocity and pressure of the VEM. To focus on the discretisation error of the right-hand side consider the following dual norms
\begin{align*}
\| \vec{F} - \vec{F}_h \|_{\vec{V}_{0,h}^\star}
  & := \sup_{\vec{v}_h \in \vec{V}_{0,h} \setminus \lbrace 0 \rbrace}
  \frac{F(\vec{v}_h) - F_h(\vec{v}_h)}{\| \nabla \vec{v}_h \|_{L^2}},\\
\| \vec{F} - \vec{F}_h \|_{\vec{V}_{h}^\star}
  & := \sup_{\vec{v}_h \in \vec{V}_{h} \setminus \lbrace 0 \rbrace}
  \frac{F(\vec{v}_h) - F_h(\vec{v}_h)}{\| \nabla \vec{v}_h \|_{L^2}}.
\end{align*}
Here \(\vec{V}_{0,h} := \lbrace \vec{v}_h \in \vec{V}_h \, : \, \mathrm{div} (\vec{v}_h) = 0 \rbrace\) denotes the subspace of divergence-free virtual
functions. The first dual norm refers to
testing only with divergence-free velocity test functions and the second dual norm
to testing with arbitrary ones that appear in a priori pressure estimates.

\begin{theorem}[A priori estimates]\label{thm:apriori_full}
  Under sufficient regularity assumptions on \(\vec{u}\) and \(p\), there holds
        \begin{align*}
          \| \nabla(\vec{u} - \vec{u}_h) \|_{L^2} & \lesssim \inf_{\vec{v}_h \in
          \vec{V}_{h}} \| \nabla( \vec{u} - \vec{v}_h) \|_{L^2}
          + \inf_{\vec{v}_h \in \vec{P}_k(\mathcal{T})} \| \nabla_h(\vec{u}
          - \vec{v}_h) \|_{L^2} + \frac{1}{\nu} \| \vec{F} - \vec{F}_h \|_{\vec{V}_{0,h}^\star}\\
          \| p - p_h \|_{L^2} & \lesssim \inf_{q_h \in Q_h} \| p - q_h \|_{L^2}
          + \nu \inf_{\vec{v}_h \in \vec{V}_{h}} \| \nabla( \vec{u} - \vec{v}_h)
          \|_{L^2} + \nu \inf_{\vec{v}_h \in \vec{P}_k(\mathcal{T})} \| \nabla_h(\vec{u} - \vec{v}_h) \|_{L^2} \\
          & \qquad \qquad + \| \vec{F} - \vec{F}_h \|_{\vec{V}_{h}^\star}
        \end{align*}
        where \(\nabla_h\) is the piecewise gradient with respect to \(\mathcal{T}\).
        Since the bestapproximtions converge optimally (\cite{MR3626409, BasicVEM}), the VEM has the optimal velocity and pressure convergence order $k$ whenever the consistency errors of the right-hand side discretisation is of the right order.
\end{theorem}
\begin{proof}
  See \cite{MR3626409, BasicVEM} and adapt to dual norms.
\end{proof}

\begin{lemma}[Right-hand side discretisation consistency error]
The consistency errors of the classical right-hand side discretisation are bounded by
\begin{align*}
  \| \vec{F} - \vec{F}_h \|_{\vec{V}_{0,h}^\star}
  \leq \| \vec{F} - \vec{F}_h \|_{\vec{V}_{h}^\star}
  \lesssim \| h_\mathcal{T} (\vec{f} - \vec{\pi}_{k-2} \vec{f}) \|_{L^2}
  \lesssim \| h_\mathcal{T}^k \mathrm{D}^{k-1} \vec{f} \|_{L^2}
\end{align*}
where the last estimate requires \(\vec{f} \in \vec{H}^{k-1}(\Omega)\) and \(\mathrm{D}^{k-1}\) collects all derivatives of order \(k-1\).
\end{lemma}
\begin{proof}
 This follows directly from the approximation properties of the \(\vec{L}^2\) bestapproximation \(\vec{\pi}_{k-2}\), see e.g.\ \cite{BasicVEM} for details.
\end{proof}

\begin{remark}[Classical VEM is not pressure-robust]
Although the virtual element method is divergence-free, it is in general
not pressure-robust with the classical right-hand side discretisation.
This drawback can be seen e.g.\ when $\vec{f} = \nabla q$ for some $q \notin P_{k-1}$ and small viscosity parameters $\nu$. Then, the method shows
a locking-phenomenon for \(\nu \rightarrow 0\) as it is also observed for classical finite element methods that are not divergence-free, see e.g.\ \cite{MR3743746,MR3683678} for a comprehensive introduction. 
The reason for that is that the operator 
\(\vec{\pi}_{k-2}\) alters the divergence and therefore destroys
the orthogonality between divergence-free functions and gradient forces.
The numerical examples below demonstrate this lack of pressure-robustness.
\end{remark}

\subsection{Extended virtual ansatz spaces}
One way of rendering the method more robust against gradient forces is to enlarge
the order of the projection $\vec{\pi}_s$ used in the right-hand side discretisation. With the so called enhanced spaces introduced in
\cite{MR3796371} it is
possible to employ $\vec{\pi}_k\vec{v}_h$ instead of $\vec{\pi}_{k-2}\vec{v}_h$.

For each $K\in\mathcal{T}$ the local enlarged virtual element space is given by
\begin{align*}
  \vec{U}_h^K :=\Big\{\,\vec{v}_h\in\vec{H}^1(K)\,:\,\vec{v}_{h}|_{\partial K}\in
    \vec{C}^0(\partial K),\,\vec{v}_{h}|_{E}\in\vec{P}_k(K)\text{ for all
  }E\in\mathcal{E}(K), &\\
 -\nu\Delta\vec{v}_h+\nabla
s\in\mathcal{G}_{k}(K)^\perp\text{ for some }s\in L^2(K),\,\mathrm{div}\,\vec{v}_h\in {P}_{k-1}(K)
&\,\Big\},
\end{align*}
where the order of the space $\mathcal{G}_s^\perp$ was
increased from $k-2$ to $k$.

This enlarged space can now be restricted to the enhanced space
\begin{align*}
W_h^K :=\Big\{\,\vec{v}_h\in\vec{U}_h^K \,:\,
\left(\vec{v}_h-\Pi_k^{\nabla,K}\vec{v}_h,\vec{g}_k^\perp\right)_{L^2(K)}=0\text{
for all }\vec{g}_k^\perp\in \mathcal{G}_k^\perp/\mathcal{G}_{k-2}^\perp\,\Big\}.
\end{align*}

This space has the interesting properties that it has the same dimension as the
classical virtual element space, but additionally allows to compute the
$L^2$-projection onto polynomials of degree $k$, see e.g.\
\cite{MR3796371,AHMAD2013376} for more details.

The discretisation of the right-hand side for the enhanced space then reads
\begin{align*}
  \vec{F}_h^e(\vec{v}_h) := \int_\Omega \vec{\pi}_{k} \vec{f} \cdot \vec{v}_h
  = \int_\Omega \vec{f} \cdot \vec{\pi}_{k} \vec{v}_h.
\end{align*}

This discretisation leads to a $(k+2)$-order consistency error, i.e.
\begin{align*}
  \| \vec{F} - \vec{F}_h^e \|_{\vec{V}_{0,h}^\star}
  \leq \| \vec{F} - \vec{F}_h^e \|_{\vec{V}_{h}^\star}
  \lesssim \| h_\mathcal{T} (\vec{f} - \vec{\pi}_{k} \vec{f}) \|_{L^2}
  \lesssim \| h_\mathcal{T}^{k+2} \mathrm{D}^{k+1} \vec{f} \|_{L^2},
\end{align*}
and hence the
velocity error can be bounded by
\begin{align*}
  \| \nabla(\vec{u} - \vec{u}_h) \|_{L^2} \lesssim \inf_{\vec{v}_h \in
  \vec{V}_{h}} \| \nabla( \vec{u} - \vec{v}_h) \|_{L^2} + \inf_{\vec{v}_h \in
\vec{P}_k(\mathcal{T})} \| \nabla_h(\vec{u} - \vec{v}_h) \|_{L^2} + \frac{1}{\nu} \|
  h_\mathcal{T}^{k+2} \mathrm{D}^{k+1} \vec{f} \|_{L^2}.
\end{align*}

\begin{remark}[Only asymptotic pressure-robustness]
As for the classical VEM the enhanced VEM is not pressure-robust. Consider again the situation
$\vec{f} = \nabla q$ for some $q \notin P_{k+1}$ and small viscosity parameters $\nu$. Then, on a fixed mesh, the method still shows the same
locking-behaviour for \(\nu \rightarrow 0\). 
However, for \(h \rightarrow 0\), the discretisation error converges with a faster rate
and renders the enhanced VEM at least asymptotically pressure-robust.

Uniform pressure-robustness, in particular on coarse grids, requires the replacement of \(\vec{\pi}_k\) by some operator that
preserves the divergence of \(\vec{v}_h\). This is the goal of the next section.
\end{remark}

\section{Divergence-preserving reconstruction operators on polygons}\label{sec:ReconstructionPolygons}
This section describes the design of a reconstruction operator
that is \(H(\mathrm{div})\)-conforming and preserves the divergence
of the virtual functions for all polygons \(K\in\mathcal{T}\).
The main idea is to employ a subtriangulation of each
polygon and to compute a suitable Raviart--Thomas interpolation on that
subtriangulation $\mathcal{T}(K)$. 

\subsection{Raviart--Thomas finite element space and interpolation}
The Raviart--Thomas finite element space of order $m$ on
a subtriangulation \(\mathcal{T}(K)\) is defined by
\begin{multline*}
  \mathrm{RT}_m(\mathcal{T}(K)):= \Bigl\lbrace
  \vec{w}_h\in\vec{P}_{m+1}(\mathcal{T}(K))\cap H(\mathrm{div},K) \,:\,
  \forall T\in\mathcal{T}(K)\,\exists\vec{a}\in \vec{P}_m(T),\,b\in P_m(T),\\
  \vec{w}_{h}|_T(\vec{x}) = \vec{a}(\vec{x}) + b(\vec{x})\vec{x}\Bigr\rbrace.
\end{multline*}
The standard Raviart-Thomas interpolation $\Pi_{\mathrm{RT}_m} \vec{v}_h \in \mathrm{RT}_m(\mathcal{T}(K))$ of some (virtual) function \(\vec{v}_h \in \vec{V}_h\) is defined by
  \begin{align*}
    \int_T (\Pi_{\mathrm{RT}_m} \vec{v}_h - \vec{v}_h) \cdot \vec{q}_h \,
    \mathrm{d}s & = 0 && \text{ for  all } \vec{q}_h \in \vec{P}_{m-1}(\mathcal{T}(P)),\\
     \int_E (\Pi_{\mathrm{RT}_m} \vec{v}_h -\vec{v}_h) \cdot \vec{n} q_h \, \mathrm{d}s & = 0 && \text{ for  all } E \in \mathcal{E}(\mathcal{T}(P)) \text{ and } q_h \in P_{m}(E).
  \end{align*} 
  Here \(\mathcal{E}(\mathcal{T}(P))\) denotes the set of edges in the subtriangulation \(\mathcal{T}(P)\). The following lemma collects the well-known properties of the Raviart--Thomas standard interpolation.

  \begin{lemma}[Properties of the Raviart--Thomas standard interpolation]
  \label{lem:approxprops_RTstandard}
  For any \(\vec{V}_h \in \vec{V}_h\), there holds
  \begin{align*}
    &\text(i) & \mathrm{div}(\Pi_{\mathrm{RT}_m}(\vec{v}_h)) & = \pi_m (\mathrm{div} (\vec{v}_h))\\
	&\text(ii) & \|  \vec{v}_h - \Pi_{\mathrm{RT}_m} \vec{v}_h \|_{L^2(K)} & \lesssim \| h_{\mathcal{T}(K)} \nabla \vec{v}_h \|_{L^2(K)} \leq h_K \|\nabla \vec{v}_h \|_{L^2(K)},\\
	&\text(iii) & \text{if } m > 0: \qquad \int_K (\vec{v}_h - \Pi_{\mathrm{RT}_m}
    \vec{v}_h) \cdot \vec{q}_h \, \mathrm{d}x & = 0
    \quad \text{for all } \vec{q}_h \in \vec{P}_{m-1}(\mathcal{T}(K)).
\end{align*}
  \end{lemma}
  \begin{proof}
    See textbooks like e.g.\ \cite{MR3097958}.
  \end{proof}
  
  Note, that the Raviart--Thomas standard interpolation of some virtual function cannot be calculated in general (see Remark~\ref{rem:ontriangles}
  for an exception on triangles). Hence, one has to devise a strategy based on the known degrees of freedom. The goal of the design below is to ensure crucial properties of the  Raviart--Thomas standard interpolation.

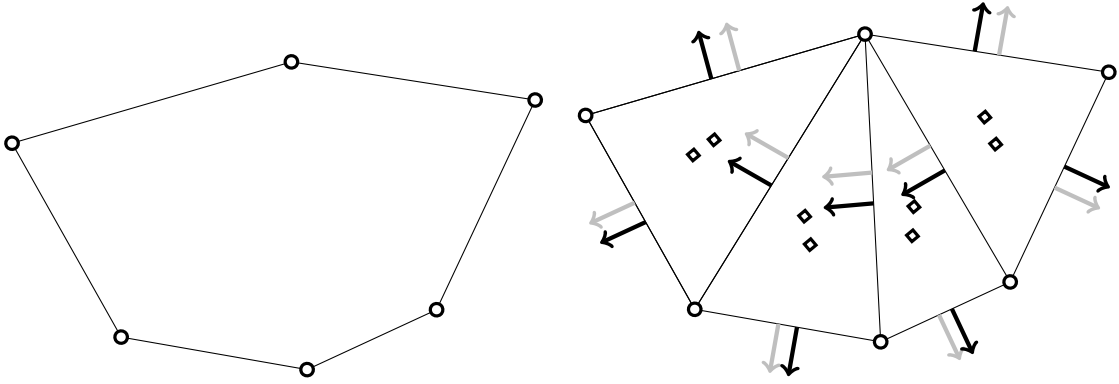
\begin{figure}
\begin{center}
\hfill
\begin{tikzpicture}[scale=0.33,rotate=130]
\path (280:2) coordinate (P1)
      (60:10) coordinate (P2)
      (110:11) coordinate (P3)
      (150:11) coordinate (P4)
      (180:11) coordinate (P5)
      (230:11) coordinate (P6);
\draw (P1)--(P2)--(P3)--(P4)--(P5)--(P6)--(P1);
\foreach \x in {P1,P2,P3,P4,P5,P6}
{%
\filldraw[fill=white,very thick] (\x) circle (0.25);
}
\end{tikzpicture}
\hfill
\begin{tikzpicture}[scale=0.33,rotate=130]
\path (280:2) coordinate (P1)
      (60:10) coordinate (P2)
      (110:11) coordinate (P3)
      (150:11) coordinate (P4)
      (180:11) coordinate (P5)
      (230:11) coordinate (P6);
\draw (P1)--(P2)--(P3)--(P4)--(P5)--(P6)--(P1);
\draw (P1)--(P3);
\draw (P1)--(P4);
\draw (P1)--(P5);

\path (P1)--(P2) coordinate[pos=0.45] (Q11)
      (P1)--(P2) coordinate[pos=0.55] (Q12)
      (P2)--(P3) coordinate[pos=0.45] (Q21)
      (P2)--(P3) coordinate[pos=0.55] (Q22)
      (P3)--(P4) coordinate[pos=0.45] (Q31)
      (P3)--(P4) coordinate[pos=0.55] (Q32)
      (P4)--(P5) coordinate[pos=0.45] (Q41)
      (P4)--(P5) coordinate[pos=0.55] (Q42)
      (P5)--(P6) coordinate[pos=0.45] (Q51)
      (P5)--(P6) coordinate[pos=0.55] (Q52)
      (P6)--(P1) coordinate[pos=0.45] (Q61)
      (P6)--(P1) coordinate[pos=0.55] (Q62)
      (P1)--(P3) coordinate[pos=0.45] (Q71)
      (P1)--(P3) coordinate[pos=0.55] (Q72)
      (P1)--(P4) coordinate[pos=0.45] (Q81)
      (P1)--(P4) coordinate[pos=0.55] (Q82)
      (P1)--(P5) coordinate[pos=0.45] (Q91)
      (P1)--(P5) coordinate[pos=0.55] (Q92)
      (Q21)--(P1) coordinate[pos=0.28] (M1)
      (Q21)--(P1) coordinate[pos=0.37] (M2)
      (Q32)--(P1) coordinate[pos=0.28] (M3)
      (Q31)--(P1) coordinate[pos=0.37] (M4)
      (Q41)--(P1) coordinate[pos=0.28] (M5)
      (Q42)--(P1) coordinate[pos=0.37] (M6)
      (Q51)--(P1) coordinate[pos=0.28] (M7)
      (Q52)--(P1) coordinate[pos=0.37] (M8);
\draw (P1)--(P2)--(P3)--(P1);
\draw[lightgray,ultra thick,->] (Q11)-- +(-25:2) coordinate (T11);
\draw[ultra thick,->] (Q12)-- +(-25:2) coordinate (T12);
\draw[lightgray,ultra thick,->] (Q21)-- +(75:2) coordinate (T21);
\draw[ultra thick,->] (Q22)-- +(75:2) coordinate (T22);
\draw[lightgray,ultra thick,->] (Q31)-- +(130:2) coordinate (T31);
\draw[ultra thick,->] (Q32)-- +(130:2) coordinate (T32);
\draw[lightgray,ultra thick,->] (Q41)-- +(165:2) coordinate (T41);
\draw[ultra thick,->] (Q42)-- +(165:2) coordinate (T42);
\draw[lightgray,ultra thick,->] (Q51)-- +(205:2) coordinate (T51);
\draw[ultra thick,->] (Q52)-- +(205:2) coordinate (T52);
\draw[lightgray,ultra thick,->] (Q61)-- +(-50:2) coordinate (T61);
\draw[ultra thick,->] (Q62)-- +(-50:2) coordinate (T62);
\draw[lightgray,ultra thick,->] (Q71)-- +(20:2) coordinate (T71);
\draw[ultra thick,->] (Q72)-- +(20:2) coordinate (T72);
\draw[lightgray,ultra thick,->] (Q81)-- +(55:2) coordinate (T81);
\draw[ultra thick,->] (Q82)-- +(55:2) coordinate (T82);
\draw[lightgray,ultra thick,->] (Q91)-- +(80:2) coordinate (T91);
\draw[ultra thick,->] (Q92)-- +(80:2) coordinate (T92);

\foreach \x in {P1,P2,P3,P4,P5,P6}
{%
\filldraw[fill=white,very thick] (\x) circle (0.25);
}

\foreach \x in {M1,M2,M3,M4,M5,M6,M7,M8}
{%
\filldraw[fill=white,very thick] (\x) rectangle ++(0.33,0.33);
}
\end{tikzpicture}
\hfill
\vspace{0.25cm}
\caption{
A polygon (left) and a possible subtriangulation (right)
and its facial (arrows) and interior (squares) degrees of freedom for the
Raviart--Thomas interpolation of order $1$. Light-gray arrows relate to the
degrees of freedom of the lowest-order Raviart--Thomas space.}
\label{fig:polygon_subtriangulation}
\end{center}
\end{figure}

\subsection{Design of reconstruction operator by local minimisation problems}
On a fixed subtriangulation \(\mathcal{T}(K)\) of a polygon \(K\) (such that no additional nodes on \(\partial K\) are introduced), the local
reconstruction of some local basis function \(\vec{v}_h\) is defined by
\begin{align}\label{eqn:reconstruction_on_polygons}
   I_{\mathrm{RT}_m}(\vec{v}_h)
   := \argmin_{\vec{w}_h \in W_h(K,\vec{v}_h,m)}
   \| \Pi^{\nabla,K}_k(\vec{v}_h) - \vec{w}_h \|_{L^2(K)}
\end{align}
where
\begin{align*}
 W_h(K,\vec{v}_h,m) := \Bigl\lbrace \vec{w}_h \in \mathrm{RT}_{m}(\mathcal{T}(K))
 : \forall q_h \in P_{m}(\mathcal{T}(K)), \ \int_{K} \mathrm{div}(\vec{v}_h - \vec{w}_h) q_h \, \mathrm{d}x & = 0 \\
 \text{and }
 \forall \vec{q}_h^\perp \in \mathcal{G}_{m-1}^\perp, \ \int_{K} (\vec{v}_h
 - \vec{w}_h) \cdot \vec{q}_h^\perp \, \mathrm{d}x & = \vec{0}  \\
  \text{ and }
  \forall E \in \mathcal{E}(K), q_h \in P_{m}(E), \ \int_E (\vec{v}_h - \vec{w}_h) \cdot \vec{n} \, q_h \, \mathrm{d}s & = 0 \Bigr\rbrace.
\end{align*}

The following lemma states that the set $W_h(K,\vec{v}_h,m)$ is non-empty, and
the minimisation problem defining \(I_{\mathrm{RT}_m}(\vec{v}_h)\) therefore is well-defined. Remark~\ref{rem:qlimit} however shows that
we have to choose \(m \leq k-1\) and the a priori error estimates in
Section~\ref{sec:aprioriprobust} show that only \(m \in \{k-2,k-1\}\) are
reasonable choices.

\begin{lemma}[$W_h(K,\vec{v}_h,m)$ is non-empty]\label{lem:non-empty}
   The piecewise standard Raviart-Thomas interpolation
   \(\Pi_{\mathrm{RT}_m} \vec{v}_h\)
   is included in \(W_h(K,\vec{v}_h,m)\). 
   Moreover, if \(K\) is a triangle, it holds $W_h(K,\vec{v}_h,m) = \lbrace \Pi_{\mathrm{RT}_m} \vec{v}_h\rbrace $.
\end{lemma}
\begin{proof}
  It suffices to show that the moments of the divergence are preserved by
  the standard interpolation. An integration by parts indeed shows, for
  any \(q_h \in P_{m}(\mathcal{T}(K)\),
  \begin{align*}
    \int_{K} \mathrm{div}(\vec{v}_h - \Pi_{\mathrm{RT}_m} \vec{v}_h) q_h \, \mathrm{d}x
    & = - \int_{K} (\vec{v}_h - \Pi_{\mathrm{RT}_m} \vec{v}_h) \cdot \nabla q_h \, \mathrm{d}x + \int_{\partial K} (\Pi_{\mathrm{RT}_m} \vec{v}_h -\vec{v}_h) \cdot \vec{n} q_h \, \mathrm{d}s\\
    & = 0
  \end{align*}
  due to \(\nabla q_h \in \vec{P}_{m-1}(\mathcal{T}(P))\)
  and the properties of the standard interpolation. For \(m = 0\), the
  first property of the standard interpolation is not available, but
  also the integral over \(K\) on the right-hand side vanishes.
  This shows $\Pi_{\mathrm{RT}_m} \vec{v}_h \in W_h(K,\vec{v}_h,m)$.
  
  On a triangle (with no interior edges), a similar backward calculation
  employing the splitting \eqref{eq:polyDecomp} shows
  that every $w_h \in W_h(K,\vec{v}_h,m) $ satisfies the properties of
  $\Pi_{\mathrm{RT}_m} \vec{v}_h$, and hence $W_h(K,\vec{v}_h,m) = \lbrace \Pi_{\mathrm{RT}_m} \vec{v}_h\rbrace $.
\end{proof}

\begin{remark}
\label{rem:ontriangles}
In general the standard interpolation $\Pi_{\mathrm{RT}_m} \vec{v}_h$ of Lemma
\ref{lem:non-empty} is not computable due to the virtuality of \(\vec{v}_h\). However, if \(K\) is a triangle the standard interpolation is  directly computable up to degree \(m \leq k-1\), due to the explanations in the next remark.
\end{remark}

\begin{remark}[Constraints are computable for \(m \leq k-1\)]
\label{rem:qlimit}
Observe, that the computation of $I_{\mathrm{RT}_m}(\vec{v}_h)$ up to degree \(m \leq k-1\) for any virtual function $\vec{v}_h\in\vec{V}_h$
is possible and only involves the evaluation of the degrees of freedom of \(\vec{v}_h\).
Indeed, the divergence is a polynomial of degree at most $k-1$ and is explicitly
available using only \textbf{D\textsubscript{V}1}, \textbf{D\textsubscript{V}2}
and \textbf{D\textsubscript{V}4} (see~\cite{DASSI2019}), and hence
\begin{align*}
\int_{K} \mathrm{div}(\vec{v}_h - \vec{w}_h) q_h \, \mathrm{d}x = 0 \text{ for  all } q_h \in P_{m}(\mathcal{T}(K))
\end{align*}
is an integral over polynomials that can be computed.

The integrals related to the space \(\mathcal{G}_{m-1}^\perp\) are also directly
available from \textbf{D\textsubscript{V}3} up to degree \(m \leq k-1\) (this condition in fact prohibits to
choose \(m\) larger than \(k-1\)).

Finally, since $\vec{v}_h$ along the boundary is a polynomial of degree at most $k$ also the boundary integral
\begin{align*}
 \int_E (\vec{w}_h \cdot \vec{n}) q_h \, \mathrm{d}s = \int_E (\vec{v}_h \cdot \vec{n}) q_h \, \mathrm{d}s \text{ for  all } E \in \mathcal{E}(\partial K) 
 \text{ and } q_h \in P_{m}(E)
\end{align*}
is computable. Please confer to \cite{masterthesis} for more details and instructions for the implementation in the case \(k=2\).
\end{remark}

\begin{remark}
It is possible to fix all degrees of freedom related to the lowest order
Raviart--Thomas functions (gray arrows in Figure~\ref{fig:polygon_subtriangulation} for $m=1$) separately via the preservation of the integral mean of the polynomial divergence of the virtual function \(\vec{V}_h\). This slightly reduces the costs of the local minimisation problems.
However, please note that in any case the costs of the local minimisation problems are
comparable to the costs of the computation of \(\Pi^{\nabla,K}_k\)
and hence do not cause severe computational overhead.
\end{remark}

\begin{remark}
It is also possible to replace the finite element spaces $\mathrm{RT}_m$
by the slightly larger Brezzi--Douglas--Marini spaces $\mathrm{BDM}_{m+1}$, which may offer a slightly better divergence-free postprocessing of the
solution \(\vec{u}_h\).
\end{remark}

The following Lemma summarises the properties that can be expected from this strategy. 

\begin{theorem}[Properties of the reconstruction]
\label{thm:properties}
  For any \(\vec{v}_h \in \vec{V}_h(K)\), there holds
  \begin{align*}
    &\text(i) & \mathrm{div}(I_{\mathrm{RT}_m}(\vec{v}_h)) & = \pi_m (\mathrm{div} (\vec{v}_h))\\
	&\text(ii) & \|  \vec{v}_h - I_{\mathrm{RT}_m} \vec{v}_h \|_{L^2(K)} & \lesssim h_K \| \nabla \vec{v}_h \|_{L^2(K)},\\
	&\text(iii) & \text{if } m > 0: \qquad \int_K (\vec{v}_h - I_{\mathrm{RT}_m}
    \vec{v}_h) \cdot \vec{q}_h \, \mathrm{d}x & = 0
    \quad \text{for all } \vec{q}_h \in \vec{P}_{m-1}(K).
\end{align*}
\end{theorem}
\begin{proof}
Property (i) directly follows from the definition of \(W_h(K,\vec{v}_h,m)\).

For the proof of (ii), consider the piecewise \(\mathrm{RT}_m\) standard
interpolation \(\Pi_{\mathrm{RT}_m} \vec{v}_h\) of \(\vec{v}_h\) on the
subtriangulation and once again note that \(\Pi_{\mathrm{RT}_m} \vec{v}_h \in
\vec{W}_h(K,\vec{v}_h,m)\). Since
\begin{align*}
  (I_{\mathrm{RT}_m} \vec{v}_h - \Pi^{\nabla,K}_k(\vec{v}_h), \vec{w}_h) = 0
  \quad \text{for all } w_h \in \vec{W}_h(K,\vec{v}_h,m)
\end{align*}
by \eqref{eqn:reconstruction_on_polygons},
we obtain for \(\vec{w_h} = I_{\mathrm{RT}_m} \vec{v}_h\)
and \(\vec{w}_h = \Pi_{\mathrm{RT}_m} \vec{v}_h\)
\begin{align*}
  \| I_{\mathrm{RT}_m} \vec{v}_h - \Pi_{\mathrm{RT}_m} \vec{v}_h \|^2_{L^2(K)}
  & = (I_{\mathrm{RT}_m} \vec{v}_h - \Pi_{\mathrm{RT}_m} \vec{v}_h, I_{\mathrm{RT}_m} \vec{v}_h - \Pi_{\mathrm{RT}_m} \vec{v}_h)\\
  & = (\Pi^{\nabla,K}_k(\vec{v}_h) - \Pi_{\mathrm{RT}_m} \vec{v}_h, I_{\mathrm{RT}_m} \vec{v}_h - \Pi_{\mathrm{RT}_m} \vec{v}_h)\\
  & \leq \| \Pi^{\nabla,K}_k(\vec{v}_h) - \Pi_{\mathrm{RT}_m} \vec{v}_h \|_{L^2(K)}  \| I_{\mathrm{RT}_m} \vec{v}_h - \Pi_{\mathrm{RT}_m} \vec{v}_h \|_{L^2(K)}.
\end{align*}
This, a triangle inequality and the first-order approximation properties of
\(\Pi^\nabla_k(\vec{v}_h)\) (see \eqref{eq:approxprop_gradientprojection}) and \(\Pi_{\mathrm{RT}_m} \vec{v}_h\) (piecewise for each subtriangle, see Lemma~\ref{lem:approxprops_RTstandard}.(ii)) show
\begin{align*}
  \| I_{\mathrm{RT}_m} \vec{v}_h - \Pi_{\mathrm{RT}_m} \vec{v}_h \|_{L^2(K)}
  \leq \| \Pi^{\nabla,K}_k(\vec{v}_h) - \Pi_{\mathrm{RT}_m} \vec{v}_h \|_{L^2(K)}
  \lesssim h_K \| \nabla \vec{v}_h \|_{L^2(K)}.
\end{align*}
Another triangle inequality gives the desired result (ii).

For the proof of (iii), consider any \(\vec{q}_h \in \vec{P}_{m-1}(K)\) and its
decomposition
\eqref{eq:polyDecomp}
into some \(r_h \in P_{m}(K)\) and \(\vec{s}_h^\perp \in \mathcal{G}_{m-1}^\perp\) such that
\begin{align*}
  \vec{q}_h = \nabla r_h + \vec{s}_h^\perp.
\end{align*}
Then, an integration by parts shows
\begin{align*}
  \int_K (\vec{v}_h - I_{\mathrm{RT}_m} \vec{v}_h) \vec{q}_h \, \mathrm{d}x
  = - \int_K \mathrm{div}(\vec{v}_h - I_{\mathrm{RT}_m} \vec{v}_h) r_h \,
  \mathrm{d}x + \int_K (\vec{v}_h - I_{\mathrm{RT}_m} \vec{v}_h) \vec{s}_h^\perp \, \mathrm{d}x.
\end{align*}
Both integrals vanish due to \(I_{\mathrm{RT}_m}\vec{v}_h \in W_h(K,\vec{v}_h,m)\).
\end{proof}

\section{Pressure-robust a priori error estimates}
\label{sec:aprioriprobust}
This section shows pressure-robust estimates for the discretisation
error of the right-hand side. Together with Theorem~\ref{thm:apriori_full}
convergence rates for the modified method can be derived.

\subsection{Estimates for \texorpdfstring{$I_{\mathrm{RT}_{k-1}}$}{IRTk-1}}
Consider the modified right-hand side discretisation
\begin{align*}
  \vec{F}_{\mathrm{RT}_{k-1}}(\vec{v}_h) := \int_\Omega \vec{f} \cdot
  I_{\mathrm{RT}_{k-1}}(\vec{v}_h) \, \mathrm{d}x.
\end{align*}

\begin{lemma}[Modified right-hand side discretisation consistency error]
The consistency error of the modified right-hand side discretisation is bounded by
\begin{align*}
  \| \vec{F} - \vec{F}_{\mathrm{RT}_{k-1}} \|_{\vec{V}_{0,h}^\star}
  & := \sup_{\vec{V}_h \in \vec{V}_h \setminus \lbrace 0 \rbrace} \frac{\vec{F}(\vec{v}_h) - \vec{F}_{\mathrm{RT}_{k-1}}(\vec{v}_h)}{\| \nabla \vec{v}_h \|_{L^2}}
  \lesssim \| h_\mathcal{T} (\mathbb{P} \vec{f} - \vec{\pi}_{k-2}(\mathbb{P} \vec{f})) \|_{L^2},\\
  \| \vec{F} - \vec{F}_{\mathrm{RT}_{k-1}} \|_{\vec{V}_{h}^\star}
  & := \sup_{\vec{V}_h \in \vec{V}_h \setminus \lbrace 0 \rbrace} \frac{\vec{F}(\vec{v}_h) - \vec{F}_{\mathrm{RT}_{k-1}}(\vec{v}_h)}{\| \nabla \vec{v}_h \|_{L^2}}
  \lesssim \| h_\mathcal{T} (\vec{f} - \vec{\pi}_{k-2}\vec{f}) \|_{L^2}
\end{align*}
where \(\vec{\pi}_{-1} \equiv 0\).
If \(\Delta \vec{u} \in H^{k-1}\), then
it holds
\begin{align*}
\| \vec{F} - \vec{F}_{\mathrm{RT}_{k-1}} \|_{\vec{V}_{0,h}^\star} \lesssim \nu \| h^k_\mathcal{T} D^{k-1} \Delta \vec{u} \|_{L^2}.
\end{align*}
\end{lemma}
\begin{proof}
  Indeed, for any (divergence-free) \(\vec{v}_h \in \vec{\vec{V}_{0,h}}\), it holds
  \begin{align*}
  \int_\Omega (\vec{f} - \mathbb{P} \vec{f}) \cdot (\vec{v}_h -  I_{\mathrm{RT}_{k-1}} \vec{v}_h) \, \mathit{dx} = 0.
  \end{align*}
  This and
  the properties (i)-(iii) of Lemma~\ref{thm:properties}  yield
  \begin{align*}
    \vec{F}(\vec{v}_h) - \vec{F}_{\mathrm{RT}_{k-1}}(\vec{v}_h)
    & = \int_\Omega \mathbb{P} \vec{f} \cdot (\vec{v}_h -  I_{\mathrm{RT}_{k-1}}
  \vec{v}_h) \, \mathrm{d}x\\
    & = \int_\Omega (\mathbb{P} \vec{f} - \vec{\pi}_{k-2} (\mathbb{P} \vec{f})) \cdot
  (\vec{v}_h -  I_{\mathrm{RT}_{k-1}} \vec{v}_h) \, \mathrm{d}x\\
    & \lesssim \sum_{K \in \mathcal{T}} \| \mathbb{P} \vec{f} - \vec{\pi}_{k-2} (\mathbb{P} \vec{f}) \|_{L^2(K)} \| \vec{v}_h -  I_{\mathrm{RT}_{k-1}} \vec{v}_h \|_{L^2(K)}\\
    &  \lesssim \sum_{K \in \mathcal{T}} h_P \| \mathbb{P} \vec{f} - \vec{\pi}_{k-2} (\mathbb{P} \vec{f}) \|_{L^2(K)} \| \nabla \vec{v}_h \|_{L^2(P)}\\
    & \leq \| h_\mathcal{T} (\mathbb{P} \vec{f} - \vec{\pi}_{k-2}(\mathbb{P} \vec{f}) \|_{L^2} \| \nabla \vec{v}_h \|_{L^2}.
  \end{align*}
  Since also
\((\mathbb{P} \vec f + \nu \Delta \vec{u}, \vec{v}_h -  I_{\mathrm{RT}_{k-1}}
  \vec{v}_h)_{L^2} = 0\), 
  the same calculation can be performed with  \(\mathbb{P} \vec f\)
  replaced by \(- \nu \Delta \vec{u}\) and leads to
\begin{align*}
\vec{F}(\vec{v}_h) - \vec{F}_{\mathrm{RT}_{k-1}}(\vec{v}_h) \lesssim \nu \| h_\mathcal{T} (\Delta \vec{u} - \vec{\pi}_{k-2}(\Delta \vec{u})) \|_{L^2} \| \nabla \vec{v}_h \|_{L^2}
\lesssim \nu \| h^k_\mathcal{T} D^{k-1} \Delta \vec{u} \|_{L^2} \| \nabla \vec{v}_h \|_{L^2}.
\end{align*}
  
  For (non divergence-free) \(\vec{v}_h \in \vec{V}_h\), one has to do the same calculation with \(\vec{f}\) instead of \(\mathbb{P} \vec{f}\).
\end{proof}
Hence, the reconstruction operator with \(m = k-1\) results in a
discretisation error of optimal order that is pressure-robust.

\subsection{Alternative estimate for \texorpdfstring{$m = 0$}{m=0}}
\label{subsec:RT0}
Consider the lowest-order interpolation
\begin{align*}
   I_{\mathrm{RT}_0}(\vec{v}_h)
   := \argmin_{\vec{w}_h \in W_h(K,\vec{v}_h,0)}
   \| \Pi^{\nabla,K}_k(\vec{v}_h) - \vec{w}_h \|_{L^2(K)}.
\end{align*}

\begin{theorem}
Given some right-hand side \(\vec{f}\) with \(\vec{f} \in H(\mathrm{curl},\Omega)\), it holds
\begin{align*}
  \| \vec{F} - \vec{F}_{\mathrm{RT}_0} \|_{\vec{V}_{0,h}^\star}
  \lesssim \| h^2_\mathcal{T} \mathrm{curl}(\vec{f}) \|_{L^2}.
\end{align*}
\end{theorem}
\begin{proof}
  The proof is based on the same idea as the proof in \cite[Theorem 7 (for \(I
  = 1\) and \(\Pi = I_{\mathrm{RT}_0}\))]{Lederer2019} . Indeed, due to \(\vec{v}_h \in
  \vec{V}_{0,h}\) and \(I_{\mathrm{RT}_0} \vec{v}_h\) being divergence-free, it holds \(\vec{v}_h
  - I_{\mathrm{RT}_0} \vec{v}_h = \mathrm{curl} \psi\) for some \(\psi \in H^1_0(\Omega) \cap
  H^2(\Omega)\), and hence
  \begin{align*}
    \vec{F}(\vec{v}_h) - \vec{F}_{\mathrm{RT}_0}(\vec{v}_h)
  & = \int_\Omega \vec{f} \cdot (\vec{v}_h - I_{\mathrm{RT}_0} \vec{v}_h) \, \mathrm{d}x\\
  & = \int_\Omega \vec{f} \cdot \mathrm{curl}(\psi - I_\mathcal{L} \psi) \, \mathrm{d}x\\
    & = -\int_\Omega \mathrm{curl} (\vec{f}) \cdot (\psi - I_\mathcal{L} \psi) \, \mathrm{d}x
    \end{align*}
  where \(I_\mathcal{L} \psi\) is the nodal interpolation due to the commuting properties of the de Rham complex (on subtriangles) \(\mathrm{curl}(I_\mathcal{L} \psi) = I_{\mathrm{RT}_0}(\mathrm{curl} \psi) = 0\). Standard elementwise interpolation estimates then result in
  \begin{align*}
    \vec{F}(\vec{v}_h) - \vec{F}_{\mathrm{RT}_0}(\vec{v}_h) & \lesssim \| h^2_\mathcal{T} \mathrm{curl}(\vec{f}) \|_{L^2} \| h^{-2}_\mathcal{T} (\psi - I_\mathcal{L} \psi) \|_{L^2}\\
    & \lesssim \| h^2_\mathcal{T} \mathrm{curl}(\vec{f}) \|_{L^2} \| h^{-1}_\mathcal{T} \nabla \psi \|_{L^2}\\
    & = \| h^2_\mathcal{T} \mathrm{curl}(\vec{f}) \|_{L^2} \| h^{-1}_\mathcal{T} \mathrm{curl} \psi \|_{L^2}\\
    & = \| h^2_\mathcal{T} \mathrm{curl}(\vec{f}) \|_{L^2} \| h^{-1}_\mathcal{T} (\vec{v}_h - I_{\mathrm{RT}_0} \vec{v}_h) \|_{L^2}\\
    & \lesssim \| h^2_\mathcal{T} \mathrm{curl}(\vec{f}) \|_{L^2} \| \nabla \vec{v}_h \|_{L^2}.
  \end{align*}  
  This concludes the proof.
\end{proof}

\begin{remark}
  In the case $k=2$ this results in a pressure-robust velocity discretisation that converges with the optimal order. The pressure error however may convergence suboptimally, since this relates to testing with non-divergence-free functions. In other words, no improved estimate for the full dual norm
\(\| \vec{F}(\vec{v}_h) - \vec{F}_{\mathrm{RT}_0}(\vec{v}_h) \|_{\vec{V}_{h}}\) is possible, which is also confirmed by the numerical experiments below.
\end{remark}

\section{Numerical experiments}\label{sec:Numerics}
This Section studies three numerical examples to confirm that
the new approach has optimal convergence rates and is really pressure-robust, opposite to the classical and the
enhanced VEM.
To conduct the experiments the lowest-order VEM with $k=2$ is implemented
allowing for a $RT_1$ and $RT_0$ reconstruction as described in Sections~\ref{sec:ReconstructionPolygons} and \ref{sec:aprioriprobust} to keep the optimal order of
convergence with respect to the velocity.

Since the discrete solution $\vec{u}_h$ is still virtual, the errors between the exact
solution $\vec{u}$ and the projection $\Pi_2^\nabla\vec{u}_h$ are computed,
i.e.
\begin{align*}
  \Vert \nabla (\vec{u}-\Pi_2^\nabla\vec{u}_h)\Vert_{L^2}.
\end{align*}
Moreover, the error will be computed on a series of meshes with different number
of degrees of freedom \texttt{ndof}
to gain convergence rates with respect to \(\texttt{ndof}^{-1/2}\).

In these examples the numerical domain $\Omega = [0,1]^2$ is partitioned into
a series of meshes $\mathcal{T}_0,\mathcal{T}_1,\mathcal{T}_2,\ldots $ with the following structure:
The unit square is divided into four parts with equal size. The first part
consists of (deterministically) distorted quadrilaterals whereas the second part is made of smaller regular squares. Triangles and non-convex pentagons are used to build the third part. Last but not least, the fourth part is constructed using L-shaped polygons
including hanging nodes (after the first refinement) and regular squares, see
figure~\ref{fig:meshes} for the first three meshes.

\begin{figure}[t]
\begin{minipage}[h]{0.32\textwidth}
\centering
\input{meshUnstructured2.tex}
\subcaption{intial mesh $\mathcal{T}_{0}$}
\end{minipage}
\hfill
\begin{minipage}[h]{0.32\textwidth}
\centering
\input{meshUnstructured3.tex}
\subcaption{first refinement $\mathcal{T}_{1}$}
\end{minipage}
\hfill
\begin{minipage}[h]{0.32\textwidth}
\centering
\input{meshUnstructured.tex}
\subcaption{second refinement $\mathcal{T}_{2}$}
\end{minipage}
\caption[]{First three levels of meshes used for the computations.}
\label{fig:meshes}
\end{figure}
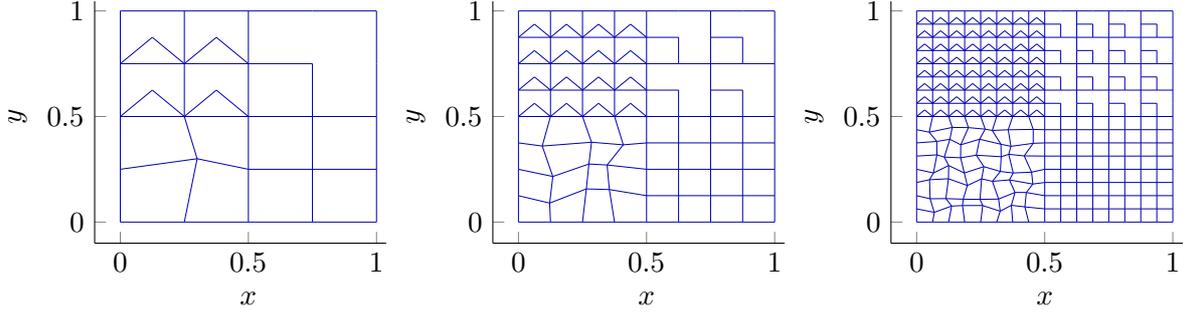

\subsection{Hydrostatic problem with different viscosities}%
\label{sub:hydrostatic_problem_with_small_viscosities}
The first experiment is performed only on the third mesh $\mathcal{T}_2$
of Figure~\ref{fig:meshes}.
\begin{figure}[t]
\begin{minipage}[h]{0.99\textwidth}
\centering
\input{legend.tex}
\end{minipage}
\vfill

\begin{minipage}[h]{0.99\textwidth}
\centering
\input{errorVsViscosity.tex}
\end{minipage}
\caption[]{Dependence of the velocity error on the viscosity computed on the
  second refinement $\mathcal{T}_{2}$ of
  $\mathcal{T}_{0}$ for different right-hand side
discretisations.}
\label{fig:hydrostatic}
\end{figure}
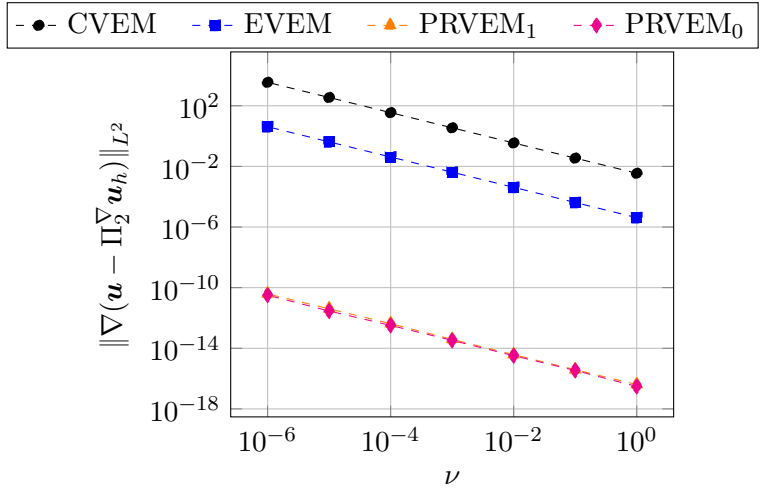

The continuous right-hand side is chosen, such that the exact solution reads
\begin{align*}
  \vec{u}=\vec{0}\in\vec{V}_h,\quad\text{and}\quad p(x,y)= \sum^{7}_{j=0}
  x^jy^{7-j}- \frac{761}{1260} \notin Q_h.
\end{align*}

To show the lack of pressure-robustness of the classical and enhanced VEM the viscosity is varied between $\nu
= 10^{0}, 10^{-1},\ldots,10^{-6}$.

All four different right-hand side discretisations are tested: The classical
virtual element method (CVEM), the enhanced virtual element method (EVEM),
the new pressure-robust version with $\mathrm{RT}_1$ reconstruction
(PRVEM\textsubscript{1}), and the pressure-robust version using the
$\mathrm{RT}_0$ reconstruction (PRVEM\textsubscript{0}).

Figure~\ref{fig:hydrostatic} shows the error of the four methods versus the viscosity and visualises the lack of pressure robustness of the classical and the enhanced VEM. The error of the pressure-robust versions are almost zero. The reason that they are not closer to machine precision is that the solver only ensures that the product of the velocity error times $\nu$ is close to machine precision.

\subsection{Vorticity problem with constant viscosity}%
\label{sub:vorticity_problem_with_constant_viscosity}
The second experiment is conducted on the series of meshes $\mathcal{T}_0,\mathcal{T}_1,\mathcal{T}_2,\ldots$ of Figure~\ref{fig:meshes} to obtain convergence rates.

The right-hand side is set in such a way that the exact solution is given by
\begin{align*}
  \vec{u}(x,y)&= \begin{pmatrix}
    - \partial/\partial y\\
    \partial/\partial x
  \end{pmatrix}
  \left(x^2(x-1)^2y^2(y-1)^2\right)\notin \vec{V}_h,
  \quad\text{and}\quad
  p(x,y)=\sin{(2\pi x)}\cos{(2\pi y)}\notin Q_h.
\end{align*}

For two viscosities $\nu = 1,$ $\nu=0.0001$ the discrete solutions
are computed on the first seven levels.
The convergence rates for the different viscosities with respect to the total number of degrees of freedom $\texttt{ndof}^{-1/2} \approx h$ are shown in Figure~\ref{fig:exp2}.

\begin{figure}[bt]
\begin{minipage}[h]{0.99\textwidth}
\centering
\input{legend.tex}
\end{minipage}
\vfill

\begin{minipage}[h]{0.49\textwidth}
\centering
\input{errVsNdof_nu=1_Un.tex}
\subcaption{moderate viscosity}
\end{minipage}
\hfill
\begin{minipage}[h]{0.49\textwidth}
\centering
\input{errVsNdof_nu=0.0001_Un.tex}
\subcaption{small viscosity}
\end{minipage}




\caption[]{Convergence rates of the velocity for the second experiment with two
different viscosities.}
\label{fig:exp2}
\end{figure}
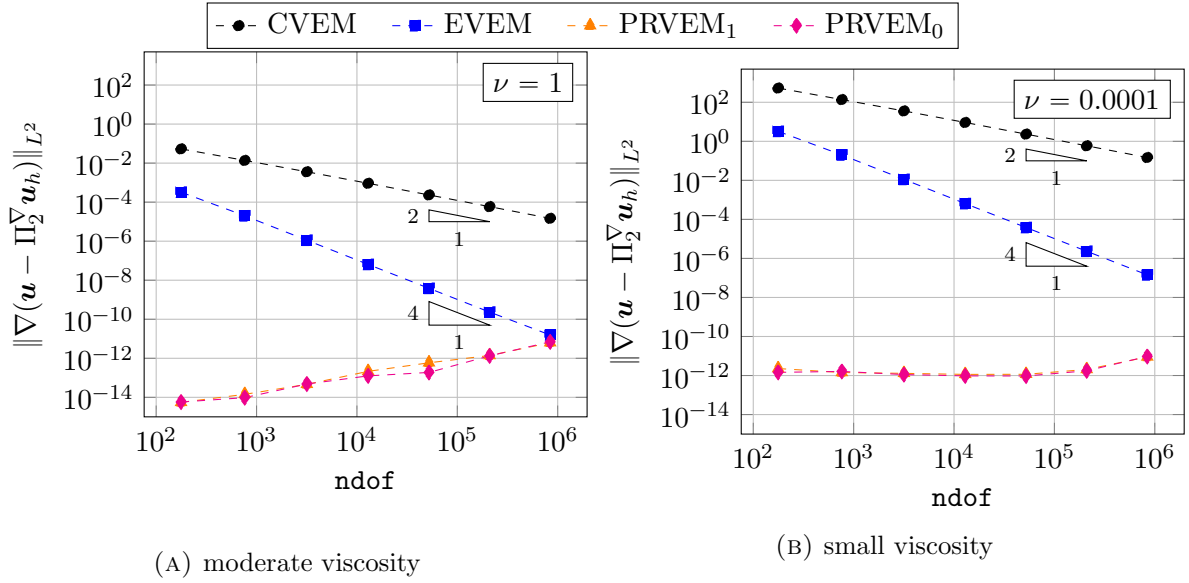
\begin{table}[b]
\centering
\caption[]{Total number of degrees of freedom, pressure error \(\| p - p_h \|_{L^2}\) and convergence rates for the
experiment of Subsection~\ref{sub:vorticity_problem_with_constant_viscosity} for viscosity $\nu = 1$.}
\label{tab:ConvRatesPressureExp2}
\begin{tabular}{c|cc|cc|cc|cc}
  &\multicolumn{2}{c|}{CVEM}&\multicolumn{2}{c|}{EVEM}
  & \multicolumn{2}{c|}{PRVEM\textsubscript{1}}
  & \multicolumn{2}{c}{PRVEM\textsubscript{0}} \\
  \texttt{ndof} &\centering error & rate &\centering  error & rate
      & error & rate &\centering  error & rate \\\hline
  177    & $ 1.939 \cdot 10^{-1}$ & -    & $ 1.367 \cdot 10^{-1} $ & -     & $ 1.367 \cdot 10^{-1} $ & -    & $ 2.099 \cdot 10^{-1} $ & - \\
  763    & $ 8.535 \cdot 10^{-2}$ & 1.12 & $ 4.656 \cdot 10^{-2} $ & 1.47 & $ 4.655 \cdot 10^{-2} $ & 1.47 & $ 1.072 \cdot 10^{-1} $ & 0.92 \\
  3171   & $ 2.407 \cdot 10^{-2}$ & 1.78 & $ 1.227 \cdot 10^{-2} $ & 1.87 & $ 1.227 \cdot 10^{-2} $ & 1.87 & $5.196 \cdot 10^{-2} $ & 1.02 \\
  12931  & $ 6.257 \cdot 10^{-3}$ & 1.92 & $ 3.126 \cdot 10^{-3} $ & 1.95 & $ 3.126 \cdot 10^{-3} $ & 1.95 & $2.576 \cdot 10^{-2} $ & 1.00 \\
  52227  & $ 1.579 \cdot 10^{-3}$ & 1.97 & $ 7.849 \cdot 10^{-4} $ & 1.98 & $ 7.848 \cdot 10^{-4} $ & 1.98 & $1.289 \cdot 10^{-2} $ & 0.99 \\
  209923 & $ 3.959 \cdot 10^{-4}$ & 1.99 & $ 1.965 \cdot 10^{-4} $ & 1.99 & $ 1.964 \cdot 10^{-4} $ & 1.99 & $6.447 \cdot 10^{-3} $ & 1.00 \\
  841731 & $ 9.907 \cdot 10^{-5}$ & 2.00 & $ 4.913 \cdot 10^{-5} $ & 2.00 & $ 4.913 \cdot 10^{-5} $ & 2.00 & $3.225 \cdot 10^{-3} $ & 1.00 \\
\end{tabular}
\end{table}

All methods converge with their expected rates. In particular, the new pressure-robust versions provide significant better
results for small viscosities compared to the classical VEM and the
enhanced version on coarse grids. As stated earlier the enhanced VEM converges
with a convergence rate of $4$ as long as the right-hand side discretisation error is dominant, and hence is asymptotically pressure-robust.

  As mentioned in Subsection \ref{subsec:RT0} the pressure computed by the
  pressure-robust VEM with $\mathrm{RT}_0$-reconstruction converges only with
  order $1$ in contrast to
  all the other discretisations which lead to
  an expected convergence rate of $2$. The pressure error and the rate for the
  different versions computed for $\nu = 1$
  can be found in Table~\ref{tab:ConvRatesPressureExp2}.

\subsection{Potential flows with different polynomial degrees}%
\label{sub:potential_flows}
As before, the third experiment is performed on the series of meshes $\mathcal{T}_0,\mathcal{T}_1,\mathcal{T}_2,\ldots$ of Figure~\ref{fig:meshes}.

The exact velocity is prescribed as a polynomial potential flow $\vec{u} = \nabla r$, i.e. the gradient of
a smooth harmonic polynomial $r \in P_{s}(\Omega)$ of degree $s$.
Then, it holds $\Delta \vec{u}= \nabla(\Delta r) = 0$ and the pressure is completely determined by the right-hand side.
To demonstrate the usefulness of the pressure-robust methods in the
Navier--Stokes setting, the right-hand side is chosen to be the convection term
\begin{align*}
  \vec{f}=(\vec{u}\cdot \nabla)\vec{u} = \nabla \left(\frac{1}{2} \lvert \vec{u} \rvert^2 \right) = \nabla p
\end{align*}
which is the gradient of a polynomial \(p := \frac{1}{2} \lvert \vec{u} \rvert^2 + C\) of degree $2(s-1)$, see e.g.\ \cite{MR3564690}.
The constant $C$ is fixed by the constraint $\int_\Omega p\,\mathrm{d}x = 0$.

As in the previous experiment, convergence rates of all methods are computed for the viscosities $\nu = 1$ and $\nu = 0.0001$.

\subsubsection{Polynomial degree \texorpdfstring{$s=2$}{s=2}}%
\label{ssub:polynomial_degree_s_1_}
The choice \(r = x^2 - y^2\) leads to the linear velocity\linebreak
 $ \vec{u}(x,y) := (2x, -2y)^T$
and the corresponding pressure and right-hand side
\begin{align*}
  p(x,y):= 2x^2+2y^2 - \frac{4}{3} ,\quad \text{and} \quad \vec{f}(x,y) =(\vec{u}\cdot
  \nabla)\vec{u}=(4x, 4y)^T.
\end{align*}

\begin{figure}[bt]
\begin{minipage}[h]{0.99\textwidth}
\centering
\input{legend.tex}
\end{minipage}
\vfill

\begin{minipage}[h]{0.49\textwidth}
\centering
\input{errVsNdof_nu=1_Un.tex}
\subcaption{moderate viscosity}
\end{minipage}
\hfill
\begin{minipage}[h]{0.49\textwidth}
\centering
\input{errVsNdof_nu=0.0001_Un.tex}
\subcaption{small viscosity}
\end{minipage}

\caption[]{Convergence rates of the velocity for the third experiment with a
  linear velocity and quadratic pressure for two different viscosities.}
  \label{fig:potflow1}
\end{figure}
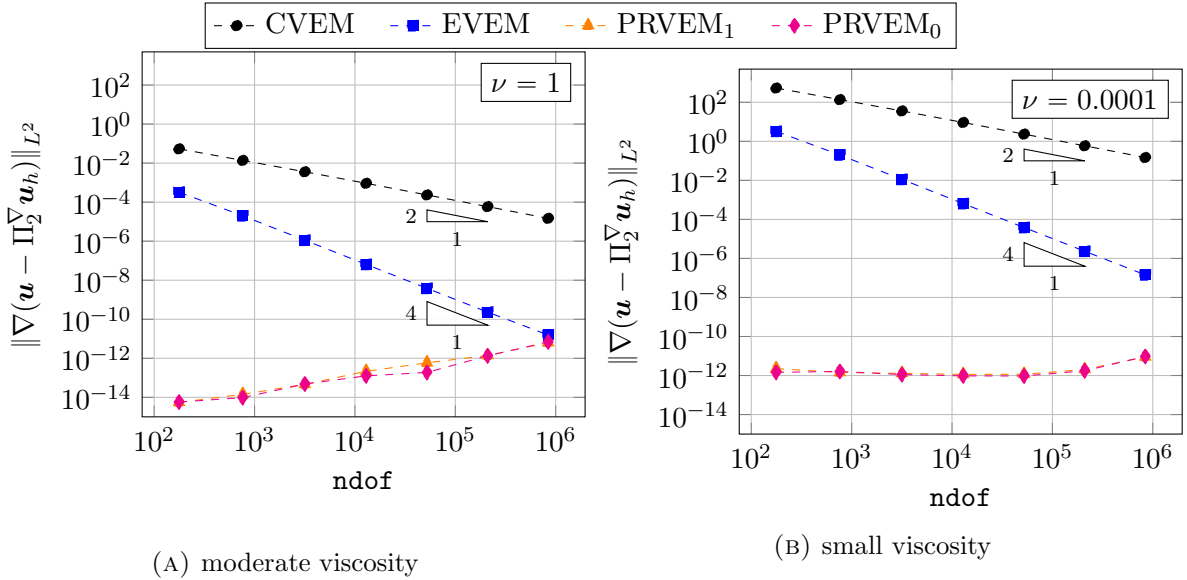

The convergence rates for all methods for the different viscosities can be found
in Figure~\ref{fig:potflow1}.
The classical method converges with its theoretically predicted order.
In this case, not only the pressure-robust versions but also the enhanced VEM can solve
the problem exact up to machine precision, since the right-hand can be
exactly approximated.

\subsubsection{Polynomial degree \texorpdfstring{$s=3$}{s=3}}%
\label{ssub:polynomial_degree_s_2_}
This time, consider $r = x^3 - 3y^2x$ and the corresponding velocity $\vec{u}:=(3x^2-3y^2, -6xy)^T$ with exact pressure and right-hand side
\begin{align*}
  p(x,y):= \frac{9}{2}(x^4+y^4)+9x^2y^2 - \frac{14}{5} ,\quad \text{and} \quad \vec{f}(x,y) =(\vec{u}\cdot
  \nabla)\vec{u}=18(x^3+xy^2, y^3+x^2y)^T
  .
\end{align*}

In Figure~\ref{fig:potflow2} the convergence rates for all methods for the
different viscosities are presented showing optimal convergence rates for the
classical and the enhanced VEM as well as the great asset of the pressure-robust
version.

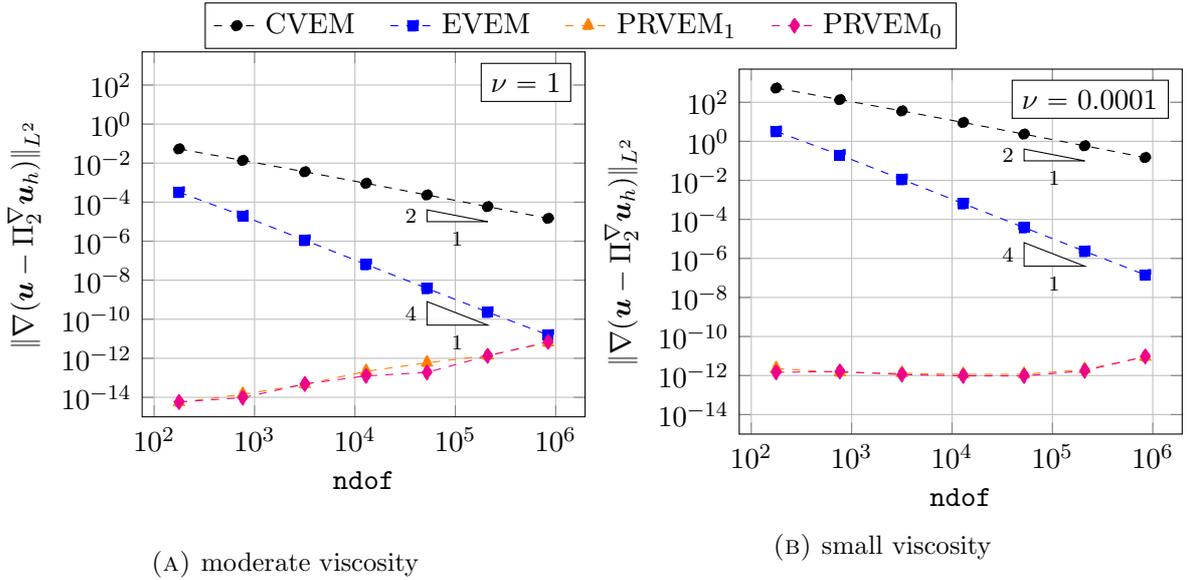
\begin{figure}[bt]
  \begin{minipage}[h]{0.99\textwidth}
\centering
\input{legend.tex}
\end{minipage}
\vfill

\begin{minipage}[h]{0.49\textwidth}
\centering
\input{errVsNdof_nu=1_Un.tex}
\subcaption{moderate viscosity}
\end{minipage}
\hfill
\begin{minipage}[h]{0.49\textwidth}
\centering
\input{errVsNdof_nu=0.0001_Un.tex}
\subcaption{small viscosity}
\end{minipage}

\caption[]{Convergence rates of the velocity for the third experiment with a
  quadratic velocity and quartic pressure for two different viscosities.}
  \label{fig:potflow2}
\end{figure}

On the other hand, this problem indicates also an advantage of
pressure-robust methods for the Navier-Stokes setting which is shortly addressed in the outlook.

\section{Outlook} \label{sec:Outlook}
This Section discusses several straight-forward extensions of the presented idea.

\subsection{Divergence-free postprocessing}
The quantity \(\Pi_k^\nabla \vec{u}_h\) is in general not divergence-free, but often used as a postprocessing to have some quantity that can be evaluated everywhere.
The reconstruction operator \(I_{\mathrm{RT}_{k-1}} \vec{u}_h\) can serve as an alternative divergence-free postprocessing of the discrete solution \(\vec{u}_h\). This might be of importance in coupled multiphysics problems to preserve structural properties like mass conservation \cite{MR3683678}.

\subsection{Extension to Navier--Stokes}
In the spirit of \cite{MR3683678,MR3564690,MR3743746}, the reconstruction operator can be also applied in the virtual element discretisation of the Navier--Stokes equations \cite{MR3796371}. Then, it appears not only in the right-hand side but also in the material derivative, i.e.\ time derivative or the nonlinear convection term. A modified computable discrete convection form might read
\begin{align*}
  c_h^+(\vec{w}_h,\vec{u}_h,\vec{v}_h)
  := \int_\Omega (I_{\mathrm{RT}_{k-1}}(\vec{w}_h) \cdot \vec{\pi}_{k-1} \nabla
  \vec{u}_h) I_{\mathrm{RT}_1}(\vec{v}_h) \, \mathrm{d}x
\end{align*}
and a modified discretisation of the time derivative is given by
\begin{align*}
  d_h^+(\vec{u}_h,\vec{v}_h)
  := \int_\Omega \frac{\mathrm{d}}{\mathrm{d}t} I_{\mathrm{RT}_{k-1}}(\vec{u}_h) \cdot
  I_{\mathrm{RT}_{k-1}}(\vec{v}_h) \, \mathrm{d}x
\end{align*}
where \(\frac{\mathrm{d}}{\mathrm{d}t}\) can be replaced by any discrete time stepping scheme.

In fact, as demonstrated in \cite{gauger:linke:schroeder:2019} for high Reynolds number flows, there are situations where the material derivative \(\vec{u}_t + \vec{u} \cdot\nabla \vec{u}\) is (close to) a gradient (in particular for \(\vec{f} = 0\) and \(\nu \rightarrow 0\)) of a possibly non-trivial pressure. Then, a discretisation of the terms in the material derivative based on the divergence-preserving reconstruction operator will be a better choice.

\subsection{Extension to other discretisation schemes on polygonal or polyhedral meshes}
In principle, a similar design of a reconstruction operator is possible for any
discretisation on polygonal or polyhedral meshes as long as there is
a discretely divergence-free constraint that is satisfied exactly. One example
on simplicial meshes can be found in \cite{MR3502564} for a discontinuous
skeletal method, where also Raviart--Thomas elements are used for a divergence-preserving reconstruction. In \cite{piatkowski2019highorder} a similar divergence-preserving postprocessing is used in a projection step of a splitting scheme. Those methods are extendable to general meshes \cite{DIPIETRO20151} and can then be reconstructed or postprocessed with the subgrid strategy presented here.

\bibliographystyle{amsplain}
\bibliography{lit}

\end{document}

%% file: meshUnstructured2.tex
\begin{tikzpicture}[]
\clip (-1.13,-.9) rectangle (3.825,3.20);
\begin{axis}[
	xtick={0,.5,...,1}, ytick={0,.5,...,1},
	axis background/.style={fill=white},
	axis x line*=bottom,axis y line*=left,
	xlabel = $x$,
	ylabel = $y$,
	scale=.590,
]\addplot[%
patch,%
darkblue,%
line width    = .1pt,%
table/row sep = \\,
	patch type = line,
	patch table={
0 13\\
0 16\\
1 17\\
1 18\\
2 20\\
2 21\\
3 23\\
3 24\\
4 14\\
4 15\\
4 19\\
4 22\\
4 26\\
5 13\\
5 14\\
5 17\\
6 18\\
6 19\\
6 20\\
7 21\\
7 22\\
7 23\\
8 15\\
8 16\\
8 24\\
8 25\\
9 13\\
9 14\\
9 15\\
9 16\\
10 14\\
10 17\\
10 18\\
10 19\\
11 19\\
11 22\\
12 15\\
12 22\\
12 23\\
12 24\\
12 27\\
12 28\\
15 25\\
15 26\\
22 27\\
24 28\\
		}]
			table[row sep=\\] {
x y\\
0.000000 0.000000\\
1.000000 0.000000\\
1.000000 1.000000\\
0.000000 1.000000\\
0.500000 0.500000\\
0.500000 0.000000\\
1.000000 0.500000\\
0.500000 1.000000\\
0.000000 0.500000\\
0.300000 0.300000\\
0.750000 0.250000\\
0.750000 0.750000\\
0.250000 0.750000\\
0.250000 0.000000\\
0.500000 0.250000\\
0.250000 0.500000\\
0.000000 0.250000\\
0.750000 0.000000\\
1.000000 0.250000\\
0.750000 0.500000\\
1.000000 0.750000\\
0.750000 1.000000\\
0.500000 0.750000\\
0.250000 1.000000\\
0.000000 0.750000\\
0.125000 0.625000\\
0.375000 0.625000\\
0.375000 0.875000\\
0.125000 0.875000\\
			};
\end{axis}
\end{tikzpicture}

%% file: meshUnstructured3.tex
\begin{tikzpicture}[]
\clip (-1.13,-.9) rectangle (3.825,3.20);
\begin{axis}[
	xtick={0,.5,...,1}, ytick={0,.5,...,1},
	axis background/.style={fill=white},
	axis x line*=bottom,axis y line*=left,
	xlabel = $x$,
	ylabel = $y$,
	scale=.590,
]\addplot[%
patch,%
darkblue,%
line width    = .1pt,%
table/row sep = \\,
	patch type = line,
	patch table={
0 41\\
0 44\\
1 56\\
1 57\\
2 68\\
2 69\\
3 79\\
3 80\\
4 48\\
4 49\\
4 62\\
4 65\\
4 86\\
5 45\\
5 46\\
5 53\\
6 59\\
6 60\\
6 66\\
7 71\\
7 72\\
7 77\\
8 51\\
8 52\\
8 75\\
8 81\\
9 42\\
9 43\\
9 47\\
9 50\\
10 54\\
10 55\\
10 58\\
10 61\\
11 63\\
11 64\\
11 67\\
11 70\\
12 73\\
12 74\\
12 76\\
12 78\\
12 89\\
12 94\\
13 41\\
13 42\\
13 45\\
14 46\\
14 47\\
14 48\\
14 55\\
15 49\\
15 50\\
15 51\\
15 73\\
15 82\\
15 85\\
16 43\\
16 44\\
16 52\\
17 53\\
17 54\\
17 56\\
18 57\\
18 58\\
18 59\\
19 60\\
19 61\\
19 62\\
19 63\\
20 66\\
20 67\\
20 68\\
21 69\\
21 70\\
21 71\\
22 64\\
22 65\\
22 72\\
22 76\\
22 90\\
23 77\\
23 78\\
23 79\\
24 74\\
24 75\\
24 80\\
24 93\\
25 41\\
25 42\\
25 43\\
25 44\\
26 42\\
26 45\\
26 46\\
26 47\\
27 47\\
27 48\\
27 49\\
27 50\\
28 43\\
28 50\\
28 51\\
28 52\\
29 46\\
29 53\\
29 54\\
29 55\\
30 54\\
30 56\\
30 57\\
30 58\\
31 58\\
31 59\\
31 60\\
31 61\\
32 48\\
32 55\\
32 61\\
32 62\\
33 62\\
33 65\\
34 60\\
34 63\\
35 67\\
35 70\\
36 64\\
36 72\\
37 51\\
37 73\\
37 74\\
37 75\\
37 83\\
37 84\\
38 49\\
38 65\\
38 73\\
38 76\\
38 87\\
38 88\\
39 72\\
39 76\\
39 77\\
39 78\\
39 91\\
39 92\\
40 74\\
40 78\\
40 79\\
40 80\\
40 95\\
40 96\\
49 85\\
49 86\\
51 81\\
51 82\\
65 87\\
72 91\\
73 83\\
73 88\\
74 93\\
74 94\\
75 84\\
76 89\\
76 90\\
78 92\\
78 95\\
80 96\\
		}]
			table[row sep=\\] {
x y\\
0.000000 0.000000\\
1.000000 0.000000\\
1.000000 1.000000\\
0.000000 1.000000\\
0.500000 0.500000\\
0.500000 0.000000\\
1.000000 0.500000\\
0.500000 1.000000\\
0.000000 0.500000\\
0.275000 0.275000\\
0.750000 0.250000\\
0.750000 0.750000\\
0.250000 0.750000\\
0.250000 0.000000\\
0.500000 0.250000\\
0.250000 0.500000\\
0.000000 0.250000\\
0.750000 0.000000\\
1.000000 0.250000\\
0.750000 0.500000\\
1.000000 0.750000\\
0.750000 1.000000\\
0.500000 0.750000\\
0.250000 1.000000\\
0.000000 0.750000\\
0.121919 0.089779\\
0.354721 0.153961\\
0.409151 0.365849\\
0.092957 0.360058\\
0.625000 0.125000\\
0.875000 0.125000\\
0.875000 0.375000\\
0.625000 0.375000\\
0.625000 0.625000\\
0.875000 0.625000\\
0.875000 0.875000\\
0.625000 0.875000\\
0.125000 0.625000\\
0.375000 0.625000\\
0.375000 0.875000\\
0.125000 0.875000\\
0.125000 0.000000\\
0.264942 0.157043\\
0.134151 0.215849\\
0.000000 0.125000\\
0.375000 0.000000\\
0.500000 0.125000\\
0.346039 0.270279\\
0.500000 0.375000\\
0.375000 0.500000\\
0.285221 0.378081\\
0.125000 0.500000\\
0.000000 0.375000\\
0.625000 0.000000\\
0.750000 0.125000\\
0.625000 0.250000\\
0.875000 0.000000\\
1.000000 0.125000\\
0.875000 0.250000\\
1.000000 0.375000\\
0.875000 0.500000\\
0.750000 0.375000\\
0.625000 0.500000\\
0.750000 0.625000\\
0.625000 0.750000\\
0.500000 0.625000\\
1.000000 0.625000\\
0.875000 0.750000\\
1.000000 0.875000\\
0.875000 1.000000\\
0.750000 0.875000\\
0.625000 1.000000\\
0.500000 0.875000\\
0.250000 0.625000\\
0.125000 0.750000\\
0.000000 0.625000\\
0.375000 0.750000\\
0.375000 1.000000\\
0.250000 0.875000\\
0.125000 1.000000\\
0.000000 0.875000\\
0.062500 0.562500\\
0.187500 0.562500\\
0.187500 0.687500\\
0.062500 0.687500\\
0.312500 0.562500\\
0.437500 0.562500\\
0.437500 0.687500\\
0.312500 0.687500\\
0.312500 0.812500\\
0.437500 0.812500\\
0.437500 0.937500\\
0.312500 0.937500\\
0.062500 0.812500\\
0.187500 0.812500\\
0.187500 0.937500\\
0.062500 0.937500\\
			};
\end{axis}
\end{tikzpicture}

%% file: meshUnstructured.tex
\begin{tikzpicture}[]
\clip (-1.13,-.9) rectangle (3.825,3.20);
\begin{axis}[
	xtick={0,.5,...,1}, ytick={0,.5,...,1},
	axis background/.style={fill=white},
	axis x line*=bottom,axis y line*=left,
	xlabel = $x$,
	ylabel = $y$,
	scale=.590,
]\addplot[%
patch,%
darkblue,%
line width    = .1pt,%
table/row sep = \\,
	patch type = line,
	patch table={
0 145\\
0 148\\
1 198\\
1 199\\
2 244\\
2 245\\
3 287\\
3 288\\
4 172\\
4 173\\
4 220\\
4 223\\
4 310\\
5 160\\
5 161\\
5 185\\
6 210\\
6 211\\
6 233\\
7 255\\
7 256\\
7 277\\
8 183\\
8 184\\
8 259\\
8 289\\
9 152\\
9 153\\
9 166\\
9 169\\
10 191\\
10 192\\
10 204\\
10 207\\
11 226\\
11 227\\
11 238\\
11 241\\
12 262\\
12 263\\
12 272\\
12 275\\
12 321\\
12 342\\
13 149\\
13 150\\
13 157\\
14 163\\
14 164\\
14 170\\
14 194\\
15 175\\
15 176\\
15 181\\
15 260\\
15 294\\
15 305\\
16 155\\
16 156\\
16 179\\
17 188\\
17 189\\
17 195\\
18 201\\
18 202\\
18 208\\
19 213\\
19 214\\
19 218\\
19 224\\
20 235\\
20 236\\
20 242\\
21 247\\
21 248\\
21 253\\
22 229\\
22 230\\
22 251\\
22 270\\
22 326\\
23 279\\
23 280\\
23 285\\
24 265\\
24 266\\
24 283\\
24 337\\
25 146\\
25 147\\
25 151\\
25 154\\
26 158\\
26 159\\
26 162\\
26 165\\
27 167\\
27 168\\
27 171\\
27 174\\
28 177\\
28 178\\
28 180\\
28 182\\
29 186\\
29 187\\
29 190\\
29 193\\
30 196\\
30 197\\
30 200\\
30 203\\
31 205\\
31 206\\
31 209\\
31 212\\
32 215\\
32 216\\
32 217\\
32 219\\
33 221\\
33 222\\
33 225\\
33 228\\
34 231\\
34 232\\
34 234\\
34 237\\
35 239\\
35 240\\
35 243\\
35 246\\
36 249\\
36 250\\
36 252\\
36 254\\
37 257\\
37 258\\
37 261\\
37 264\\
37 297\\
37 302\\
38 267\\
38 268\\
38 269\\
38 271\\
38 313\\
38 318\\
39 273\\
39 274\\
39 276\\
39 278\\
39 329\\
39 334\\
40 281\\
40 282\\
40 284\\
40 286\\
40 345\\
40 350\\
41 145\\
41 146\\
41 149\\
42 150\\
42 151\\
42 152\\
42 159\\
43 153\\
43 154\\
43 155\\
43 177\\
44 147\\
44 148\\
44 156\\
45 157\\
45 158\\
45 160\\
46 161\\
46 162\\
46 163\\
46 187\\
47 164\\
47 165\\
47 166\\
47 167\\
48 170\\
48 171\\
48 172\\
48 216\\
49 173\\
49 174\\
49 175\\
49 267\\
49 306\\
49 309\\
50 168\\
50 169\\
50 176\\
50 180\\
51 181\\
51 182\\
51 183\\
51 257\\
51 290\\
51 293\\
52 178\\
52 179\\
52 184\\
53 185\\
53 186\\
53 188\\
54 189\\
54 190\\
54 191\\
54 197\\
55 192\\
55 193\\
55 194\\
55 215\\
56 195\\
56 196\\
56 198\\
57 199\\
57 200\\
57 201\\
58 202\\
58 203\\
58 204\\
58 205\\
59 208\\
59 209\\
59 210\\
60 211\\
60 212\\
60 213\\
60 231\\
61 206\\
61 207\\
61 214\\
61 217\\
62 218\\
62 219\\
62 220\\
62 221\\
63 224\\
63 225\\
63 226\\
63 232\\
64 227\\
64 228\\
64 229\\
64 249\\
65 222\\
65 223\\
65 230\\
65 269\\
65 314\\
66 233\\
66 234\\
66 235\\
67 236\\
67 237\\
67 238\\
67 239\\
68 242\\
68 243\\
68 244\\
69 245\\
69 246\\
69 247\\
70 240\\
70 241\\
70 248\\
70 252\\
71 253\\
71 254\\
71 255\\
72 250\\
72 251\\
72 256\\
72 276\\
72 330\\
73 260\\
73 261\\
73 262\\
73 268\\
73 298\\
73 317\\
74 263\\
74 264\\
74 265\\
74 281\\
74 338\\
74 341\\
75 258\\
75 259\\
75 266\\
75 301\\
76 270\\
76 271\\
76 272\\
76 273\\
76 322\\
76 325\\
77 277\\
77 278\\
77 279\\
78 274\\
78 275\\
78 280\\
78 284\\
78 333\\
78 346\\
79 285\\
79 286\\
79 287\\
80 282\\
80 283\\
80 288\\
80 349\\
81 145\\
81 146\\
81 147\\
81 148\\
82 146\\
82 149\\
82 150\\
82 151\\
83 151\\
83 152\\
83 153\\
83 154\\
84 147\\
84 154\\
84 155\\
84 156\\
85 150\\
85 157\\
85 158\\
85 159\\
86 158\\
86 160\\
86 161\\
86 162\\
87 162\\
87 163\\
87 164\\
87 165\\
88 152\\
88 159\\
88 165\\
88 166\\
89 166\\
89 167\\
89 168\\
89 169\\
90 164\\
90 167\\
90 170\\
90 171\\
91 171\\
91 172\\
91 173\\
91 174\\
92 168\\
92 174\\
92 175\\
92 176\\
93 155\\
93 177\\
93 178\\
93 179\\
94 153\\
94 169\\
94 177\\
94 180\\
95 176\\
95 180\\
95 181\\
95 182\\
96 178\\
96 182\\
96 183\\
96 184\\
97 161\\
97 185\\
97 186\\
97 187\\
98 186\\
98 188\\
98 189\\
98 190\\
99 190\\
99 191\\
99 192\\
99 193\\
100 163\\
100 187\\
100 193\\
100 194\\
101 189\\
101 195\\
101 196\\
101 197\\
102 196\\
102 198\\
102 199\\
102 200\\
103 200\\
103 201\\
103 202\\
103 203\\
104 191\\
104 197\\
104 203\\
104 204\\
105 204\\
105 205\\
105 206\\
105 207\\
106 202\\
106 205\\
106 208\\
106 209\\
107 209\\
107 210\\
107 211\\
107 212\\
108 206\\
108 212\\
108 213\\
108 214\\
109 170\\
109 194\\
109 215\\
109 216\\
110 192\\
110 207\\
110 215\\
110 217\\
111 214\\
111 217\\
111 218\\
111 219\\
112 172\\
112 216\\
112 219\\
112 220\\
113 220\\
113 223\\
114 218\\
114 221\\
115 225\\
115 228\\
116 222\\
116 230\\
117 213\\
117 224\\
118 211\\
118 231\\
119 234\\
119 237\\
120 226\\
120 232\\
121 238\\
121 241\\
122 236\\
122 239\\
123 243\\
123 246\\
124 240\\
124 248\\
125 229\\
125 251\\
126 227\\
126 249\\
127 252\\
127 254\\
128 250\\
128 256\\
129 183\\
129 257\\
129 258\\
129 259\\
129 291\\
129 292\\
130 181\\
130 257\\
130 260\\
130 261\\
130 295\\
130 296\\
131 261\\
131 262\\
131 263\\
131 264\\
131 299\\
131 300\\
132 258\\
132 264\\
132 265\\
132 266\\
132 303\\
132 304\\
133 175\\
133 260\\
133 267\\
133 268\\
133 307\\
133 308\\
134 173\\
134 223\\
134 267\\
134 269\\
134 311\\
134 312\\
135 230\\
135 269\\
135 270\\
135 271\\
135 315\\
135 316\\
136 262\\
136 268\\
136 271\\
136 272\\
136 319\\
136 320\\
137 272\\
137 273\\
137 274\\
137 275\\
137 323\\
137 324\\
138 251\\
138 270\\
138 273\\
138 276\\
138 327\\
138 328\\
139 256\\
139 276\\
139 277\\
139 278\\
139 331\\
139 332\\
140 274\\
140 278\\
140 279\\
140 280\\
140 335\\
140 336\\
141 265\\
141 281\\
141 282\\
141 283\\
141 339\\
141 340\\
142 263\\
142 275\\
142 281\\
142 284\\
142 343\\
142 344\\
143 280\\
143 284\\
143 285\\
143 286\\
143 347\\
143 348\\
144 282\\
144 286\\
144 287\\
144 288\\
144 351\\
144 352\\
173 309\\
173 310\\
175 305\\
175 306\\
181 293\\
181 294\\
183 289\\
183 290\\
223 311\\
230 315\\
251 327\\
256 331\\
257 291\\
257 296\\
258 301\\
258 302\\
259 292\\
260 295\\
260 308\\
261 297\\
261 298\\
262 299\\
262 320\\
263 341\\
263 342\\
264 300\\
264 303\\
265 337\\
265 338\\
266 304\\
267 307\\
267 312\\
268 317\\
268 318\\
269 313\\
269 314\\
270 325\\
270 326\\
271 316\\
271 319\\
272 321\\
272 322\\
273 323\\
273 328\\
274 333\\
274 334\\
275 324\\
275 343\\
276 329\\
276 330\\
278 332\\
278 335\\
280 336\\
280 347\\
281 339\\
281 344\\
282 349\\
282 350\\
283 340\\
284 345\\
284 346\\
286 348\\
286 351\\
288 352\\
		}]
			table[row sep=\\] {
x y\\
0.000000 0.000000\\
1.000000 0.000000\\
1.000000 1.000000\\
0.000000 1.000000\\
0.500000 0.500000\\
0.500000 0.000000\\
1.000000 0.500000\\
0.500000 1.000000\\
0.000000 0.500000\\
0.262500 0.262500\\
0.750000 0.250000\\
0.750000 0.750000\\
0.250000 0.750000\\
0.250000 0.000000\\
0.500000 0.250000\\
0.250000 0.500000\\
0.000000 0.250000\\
0.750000 0.000000\\
1.000000 0.250000\\
0.750000 0.500000\\
1.000000 0.750000\\
0.750000 1.000000\\
0.500000 0.750000\\
0.250000 1.000000\\
0.000000 0.750000\\
0.114201 0.111004\\
0.383921 0.140262\\
0.368104 0.358723\\
0.129758 0.392025\\
0.625000 0.125000\\
0.875000 0.125000\\
0.875000 0.375000\\
0.625000 0.375000\\
0.625000 0.625000\\
0.875000 0.625000\\
0.875000 0.875000\\
0.625000 0.875000\\
0.125000 0.625000\\
0.375000 0.625000\\
0.375000 0.875000\\
0.125000 0.875000\\
0.125000 0.000000\\
0.247459 0.107506\\
0.125283 0.267675\\
0.000000 0.125000\\
0.375000 0.000000\\
0.500000 0.125000\\
0.376979 0.232433\\
0.500000 0.375000\\
0.375000 0.500000\\
0.245791 0.392169\\
0.125000 0.500000\\
0.000000 0.375000\\
0.625000 0.000000\\
0.750000 0.125000\\
0.625000 0.250000\\
0.875000 0.000000\\
1.000000 0.125000\\
0.875000 0.250000\\
1.000000 0.375000\\
0.875000 0.500000\\
0.750000 0.375000\\
0.625000 0.500000\\
0.750000 0.625000\\
0.625000 0.750000\\
0.500000 0.625000\\
1.000000 0.625000\\
0.875000 0.750000\\
1.000000 0.875000\\
0.875000 1.000000\\
0.750000 0.875000\\
0.625000 1.000000\\
0.500000 0.875000\\
0.250000 0.625000\\
0.125000 0.750000\\
0.000000 0.625000\\
0.375000 0.750000\\
0.375000 1.000000\\
0.250000 0.875000\\
0.125000 1.000000\\
0.000000 0.875000\\
0.068870 0.046010\\
0.179073 0.078040\\
0.197845 0.173165\\
0.050407 0.200394\\
0.326143 0.051258\\
0.422532 0.071905\\
0.453548 0.180086\\
0.295636 0.192801\\
0.329904 0.309399\\
0.419843 0.313350\\
0.455121 0.438915\\
0.295205 0.433843\\
0.079186 0.318339\\
0.171698 0.304576\\
0.202159 0.447380\\
0.049225 0.425827\\
0.562500 0.062500\\
0.687500 0.062500\\
0.687500 0.187500\\
0.562500 0.187500\\
0.812500 0.062500\\
0.937500 0.062500\\
0.937500 0.187500\\
0.812500 0.187500\\
0.812500 0.312500\\
0.937500 0.312500\\
0.937500 0.437500\\
0.812500 0.437500\\
0.562500 0.312500\\
0.687500 0.312500\\
0.687500 0.437500\\
0.562500 0.437500\\
0.562500 0.562500\\
0.687500 0.562500\\
0.687500 0.687500\\
0.562500 0.687500\\
0.812500 0.562500\\
0.937500 0.562500\\
0.937500 0.687500\\
0.812500 0.687500\\
0.812500 0.812500\\
0.937500 0.812500\\
0.937500 0.937500\\
0.812500 0.937500\\
0.562500 0.812500\\
0.687500 0.812500\\
0.687500 0.937500\\
0.562500 0.937500\\
0.062500 0.562500\\
0.187500 0.562500\\
0.187500 0.687500\\
0.062500 0.687500\\
0.312500 0.562500\\
0.437500 0.562500\\
0.437500 0.687500\\
0.312500 0.687500\\
0.312500 0.812500\\
0.437500 0.812500\\
0.437500 0.937500\\
0.312500 0.937500\\
0.062500 0.812500\\
0.187500 0.812500\\
0.187500 0.937500\\
0.062500 0.937500\\
0.062500 0.000000\\
0.136673 0.075775\\
0.052620 0.110341\\
0.000000 0.062500\\
0.187500 0.000000\\
0.257924 0.078302\\
0.181661 0.108314\\
0.253657 0.204795\\
0.186085 0.232379\\
0.124150 0.205157\\
0.065601 0.232596\\
0.000000 0.187500\\
0.312500 0.000000\\
0.369699 0.079364\\
0.319914 0.108952\\
0.437500 0.000000\\
0.500000 0.062500\\
0.428095 0.139968\\
0.500000 0.187500\\
0.448742 0.236357\\
0.362106 0.199593\\
0.326835 0.239655\\
0.359460 0.320927\\
0.328990 0.368630\\
0.232831 0.316709\\
0.500000 0.312500\\
0.455067 0.373021\\
0.500000 0.437500\\
0.437500 0.500000\\
0.357325 0.437217\\
0.312500 0.500000\\
0.267494 0.440041\\
0.107975 0.307742\\
0.078777 0.381896\\
0.000000 0.312500\\
0.172238 0.366079\\
0.187500 0.500000\\
0.138996 0.448299\\
0.062500 0.500000\\
0.000000 0.437500\\
0.562500 0.000000\\
0.625000 0.062500\\
0.562500 0.125000\\
0.687500 0.000000\\
0.750000 0.062500\\
0.687500 0.125000\\
0.750000 0.187500\\
0.687500 0.250000\\
0.625000 0.187500\\
0.562500 0.250000\\
0.812500 0.000000\\
0.875000 0.062500\\
0.812500 0.125000\\
0.937500 0.000000\\
1.000000 0.062500\\
0.937500 0.125000\\
1.000000 0.187500\\
0.937500 0.250000\\
0.875000 0.187500\\
0.812500 0.250000\\
0.875000 0.312500\\
0.812500 0.375000\\
0.750000 0.312500\\
1.000000 0.312500\\
0.937500 0.375000\\
1.000000 0.437500\\
0.937500 0.500000\\
0.875000 0.437500\\
0.812500 0.500000\\
0.750000 0.437500\\
0.625000 0.312500\\
0.562500 0.375000\\
0.687500 0.375000\\
0.687500 0.500000\\
0.625000 0.437500\\
0.562500 0.500000\\
0.625000 0.562500\\
0.562500 0.625000\\
0.500000 0.562500\\
0.750000 0.562500\\
0.687500 0.625000\\
0.750000 0.687500\\
0.687500 0.750000\\
0.625000 0.687500\\
0.562500 0.750000\\
0.500000 0.687500\\
0.875000 0.562500\\
0.812500 0.625000\\
1.000000 0.562500\\
0.937500 0.625000\\
1.000000 0.687500\\
0.937500 0.750000\\
0.875000 0.687500\\
0.812500 0.750000\\
0.875000 0.812500\\
0.812500 0.875000\\
0.750000 0.812500\\
1.000000 0.812500\\
0.937500 0.875000\\
1.000000 0.937500\\
0.937500 1.000000\\
0.875000 0.937500\\
0.812500 1.000000\\
0.750000 0.937500\\
0.625000 0.812500\\
0.562500 0.875000\\
0.500000 0.812500\\
0.687500 0.875000\\
0.687500 1.000000\\
0.625000 0.937500\\
0.562500 1.000000\\
0.500000 0.937500\\
0.125000 0.562500\\
0.062500 0.625000\\
0.000000 0.562500\\
0.250000 0.562500\\
0.187500 0.625000\\
0.250000 0.687500\\
0.187500 0.750000\\
0.125000 0.687500\\
0.062500 0.750000\\
0.000000 0.687500\\
0.375000 0.562500\\
0.312500 0.625000\\
0.437500 0.625000\\
0.437500 0.750000\\
0.375000 0.687500\\
0.312500 0.750000\\
0.375000 0.812500\\
0.312500 0.875000\\
0.250000 0.812500\\
0.437500 0.875000\\
0.437500 1.000000\\
0.375000 0.937500\\
0.312500 1.000000\\
0.250000 0.937500\\
0.125000 0.812500\\
0.062500 0.875000\\
0.000000 0.812500\\
0.187500 0.875000\\
0.187500 1.000000\\
0.125000 0.937500\\
0.062500 1.000000\\
0.000000 0.937500\\
0.031250 0.531250\\
0.093750 0.531250\\
0.093750 0.593750\\
0.031250 0.593750\\
0.156250 0.531250\\
0.218750 0.531250\\
0.218750 0.593750\\
0.156250 0.593750\\
0.156250 0.656250\\
0.218750 0.656250\\
0.218750 0.718750\\
0.156250 0.718750\\
0.031250 0.656250\\
0.093750 0.656250\\
0.093750 0.718750\\
0.031250 0.718750\\
0.281250 0.531250\\
0.343750 0.531250\\
0.343750 0.593750\\
0.281250 0.593750\\
0.406250 0.531250\\
0.468750 0.531250\\
0.468750 0.593750\\
0.406250 0.593750\\
0.406250 0.656250\\
0.468750 0.656250\\
0.468750 0.718750\\
0.406250 0.718750\\
0.281250 0.656250\\
0.343750 0.656250\\
0.343750 0.718750\\
0.281250 0.718750\\
0.281250 0.781250\\
0.343750 0.781250\\
0.343750 0.843750\\
0.281250 0.843750\\
0.406250 0.781250\\
0.468750 0.781250\\
0.468750 0.843750\\
0.406250 0.843750\\
0.406250 0.906250\\
0.468750 0.906250\\
0.468750 0.968750\\
0.406250 0.968750\\
0.281250 0.906250\\
0.343750 0.906250\\
0.343750 0.968750\\
0.281250 0.968750\\
0.031250 0.781250\\
0.093750 0.781250\\
0.093750 0.843750\\
0.031250 0.843750\\
0.156250 0.781250\\
0.218750 0.781250\\
0.218750 0.843750\\
0.156250 0.843750\\
0.156250 0.906250\\
0.218750 0.906250\\
0.218750 0.968750\\
0.156250 0.968750\\
0.031250 0.906250\\
0.093750 0.906250\\
0.093750 0.968750\\
0.031250 0.968750\\
			};
\end{axis}
\end{tikzpicture}

%% file: legend.tex
\begin{tikzpicture}[]
\begin{loglogaxis}[
	legend style ={legend columns=-1,at={(73,0.00)},anchor=north,},
  grid=none,
	xmin = .1, xmax = 1,
	ymin = .1, ymax = 1,
	scale = 0.01,
	axis background/.style={fill=white},
axis line style={draw=none},
ytick=\empty,
xtick=\empty,
]
  \addlegendimage{black,mark=*,dashed}
\addlegendentry{CVEM$\quad$}
  \addlegendimage{blue,mark=square*,dashed}
\addlegendentry{EVEM$\quad$}
  \addlegendimage{orange,mark=triangle*,dashed,mark size=2.8pt}
\addlegendentry{PRVEM\textsubscript{1}$\quad$}
  \addlegendimage{magenta,mark=diamond*,dashed,mark size=2.8pt}
\addlegendentry{PRVEM\textsubscript{0}}

\addplot[white, mark =diamond*, mark size =0.01pt] coordinates{
    (1, 1)
  };
\end{loglogaxis}

	\end{tikzpicture}

%% file: errorVsViscosity.tex
\begin{tikzpicture}[]
\begin{loglogaxis}[
	xlabel = $\nu$, ylabel = $\Vert \nabla (\vec{u}-\Pi_2^\nabla\vec{u}_h)\Vert_{L^2}$,
  grid=major,
	xtickten={-6,-4,...,0},
	ytickten={-18,-14,...,4},
	scale = 0.85,
	axis background/.style={fill=white},
]

\addplot[black, mark =*, dashed] coordinates{
    (1.0000000, 0.003520041966177529270)
    (0.1000000, 0.035200419661775304847)
    (0.0100000, 0.352004196617752951326)
    (0.0010000, 3.520041966177523740100)
    (0.0001000, 35.200419661775235624646)
    (0.0000100, 352.004196617753279952012)
    (0.0000010, 3520.041966177531321591232)
  };
\addplot[blue, mark =square*,dashed] coordinates{
    (1.0000000, 0.000004156042764775312)
    (0.1000000, 0.000041560427647742036)
    (0.0100000, 0.000415604276477551167)
    (0.0010000, 0.004156042764774126388)
    (0.0001000, 0.041560427647742211044)
    (0.0000100, 0.415604276477433531856)
    (0.0000010, 4.156042764772785780281)
  };
\addplot[orange, mark =triangle*, mark size =2.8pt,dashed] coordinates{
    (1.0000000, 0.000000000000000040394)
    (0.1000000, 0.000000000000000386485)
    (0.0100000, 0.000000000000003798741)
    (0.0010000, 0.000000000000038609888)
    (0.0001000, 0.000000000000422099640)
    (0.0000100, 0.000000000004007088698)
    (0.0000010, 0.000000000038623635992)
  };
\addplot[magenta, mark =diamond*, mark size =2.8pt,dashed] coordinates{
    (1.0000000, 0.000000000000000030196)
    (0.1000000, 0.000000000000000344761)
    (0.0100000, 0.000000000000003227088)
    (0.0010000, 0.000000000000033153613)
    (0.0001000, 0.000000000000323200066)
    (0.0000100, 0.000000000002862669720)
    (0.0000010, 0.000000000031796614752)
  };
\end{loglogaxis}

	\end{tikzpicture}

%% file: errVsNdof_nu=1_Un.tex
\begin{tikzpicture}[]
\begin{loglogaxis}[
  xlabel = \texttt{ndof}, ylabel = $\Vert \nabla (\vec{u}-\Pi_2^\nabla\vec{u}_h)\Vert_{L^2}$,
	legend style ={legend pos =north east},
  grid=major,
	ymin = 1e-15, ymax = 5e3,
	ytickten={-16,-14,...,6},
  minor y tick style = white,
  minor x tick style = white,
	scale = 0.85,
	axis background/.style={fill=white},
]
\addplot[black, mark =*, dashed] coordinates{
	(177   , 0.053546040257895513581)
	(763   , 0.013725694607277599824)
	(3171  , 0.003597476162269952309)
	(12931 , 0.000922973026567141933)
	(52227 , 0.000233802762075492170)
	(209923, 0.000058856967276728021)
	(841731, 0.000014766713047538734)
};
\addplot[blue, mark =square*,dashed] coordinates{
	(177   , 0.000324419231041478103)
	(763   , 0.000019982159386166272)
	(3171  , 0.000001120104052281499)
	(12931 , 0.000000064489274332437)
	(52227 , 0.000000003833444199779)
	(209923, 0.000000000232952661399)
	(841731, 0.000000000015633032160)
};
\addplot[orange, mark =triangle*, mark size =2.8pt,dashed] coordinates{
	(177   , 0.000000000000005511657)
	(763   , 0.000000000000014229010)
	(3171  , 0.000000000000046282274)
	(12931 , 0.000000000000214333896)
	(52227 , 0.000000000000587901492)
	(209923, 0.000000000001390576695)
	(841731, 0.000000000006291321391)
};
\addplot[magenta, mark =diamond*, mark size =2.8pt,dashed] coordinates{
	(177   , 0.000000000000005746897)
	(763   , 0.000000000000009626778)
	(3171  , 0.000000000000048109471)
	(12931 , 0.000000000000125628363)
	(52227 , 0.000000000000188612879)
	(209923, 0.000000000001331218731)
	(841731, 0.000000000006956203077)
};
	\draw[slopetriangle]
		(axis cs: 209923,0.00001)
		-- (axis cs: 52227,4.0194e-05)
		-- (axis cs: 52227,0.00001) node [midway,left]{\scriptsize\(2\)}
		-- cycle node [midway,below]{\scriptsize\(1\)};
	\draw[slopetriangle]
		(axis cs: 209923,5e-11)
		-- (axis cs: 52227,8.0779e-10)
		-- (axis cs: 52227,5e-11) node [midway,left]{\scriptsize\(4\)}
		-- cycle node [midway,below]{\scriptsize\(1\)};
	\end{loglogaxis}

  \node[draw,fill=white] at (5.05,4.435) {{$\nu = 1$}};

	\end{tikzpicture}

%% file: errVsNdof_nu=0.0001_Un.tex
\begin{tikzpicture}[]
\begin{loglogaxis}[
  xlabel = \texttt{ndof}, ylabel = $\Vert \nabla (\vec{u}-\Pi_2^\nabla\vec{u}_h)\Vert_{L^2}$,
	legend style ={legend pos =north east},
  grid=major,
	ymin = 1e-15, ymax = 5e3,
	ytickten={-16,-14,...,6},
  minor y tick style = white,
  minor x tick style = white,
	scale = 0.85,
	axis background/.style={fill=white},
]
\addplot[black, mark =*, dashed] coordinates{
	(177   , 535.460402578922639804658)
	(763   , 137.256946072764776545228)
	(3171  , 35.974761622645722525249)
	(12931 , 9.229730265872984062980)
	(52227 , 2.338027621494903840471)
	(209923, 0.588569675203101216887)
	(841731, 0.147667131210844237987)
};
\addplot[blue, mark =square*,dashed] coordinates{
	(177   , 3.244192310411761148714)
	(763   , 0.199821593856388368682)
	(3171  , 0.011201040500879526005)
	(12931 , 0.000644892781389118440)
	(52227 , 0.000038334497097490282)
	(209923, 0.000002329565362218527)
	(841731, 0.000000143437689334260)
};
\addplot[orange, mark =triangle*, mark size =2.8pt,dashed] coordinates{
	(177   , 0.000000000002278929911)
	(763   , 0.000000000001425616135)
	(3171  , 0.000000000001294124924)
	(12931 , 0.000000000001142703593)
	(52227 , 0.000000000001156056430)
	(209923, 0.000000000002042417644)
	(841731, 0.000000000008942530326)
};
\addplot[magenta, mark =diamond*, mark size =2.8pt,dashed] coordinates{
	(177   , 0.000000000001453999509)
	(763   , 0.000000000001623249671)
	(3171  , 0.000000000001121609086)
	(12931 , 0.000000000000958963292)
	(52227 , 0.000000000000949740615)
	(209923, 0.000000000001747777618)
	(841731, 0.000000000009504916807)
};
	\draw[slopetriangle]
		(axis cs: 209923,4e-7)
		-- (axis cs: 52227,6.4623e-6)
		-- (axis cs: 52227,4e-7) node [midway,left]{\scriptsize\(4\)}
		-- cycle node [midway,below]{\scriptsize\(1\)};
	\draw[slopetriangle]
		(axis cs: 209923,1e-1)
		-- (axis cs: 52227,0.4019)
		-- (axis cs: 52227,1e-1) node [midway,left]{\scriptsize\(2\)}
		-- cycle node [midway,below]{\scriptsize\(1\)};
	\end{loglogaxis}
  \node[draw,fill=white, anchor=west] at (3.585,4.44) {{$\nu = 0.0001$}};

	\end{tikzpicture}